\documentclass[10pt, reqno]{amsart}

\def\BibTeX{{\rm B\kern-.05em{\sc i\kern-.025em b}\kern-.08em
    T\kern-.1667em\lower.7ex\hbox{E}\kern-.125emX}}

\newtheorem{thm}{Theorem}[section]

\newtheorem{lem}[thm]{Lemma}
\newtheorem{prop}[thm]{Proposition}
\newtheorem{cond}[thm]{Condition}

\theoremstyle{definition}

\theoremstyle{remark}
\newtheorem{rem}{Remark}[section]
\newtheorem{exmp}{Example}[section]

\numberwithin{equation}{section}

    \newcommand{\floor}[1]{\lfloor#1\rfloor}

    \newcommand{\EE}{\mathbb{E}}

    \newcommand{\Exp}{\operatorname{E}}
    \newcommand{\E}{\Exp}

    \renewcommand{\Pr}{\operatorname{P}}

    \newcommand{\dto}{\xrightarrow{d}}
    \newcommand{\wto}{\xrightarrow{w}}
    \newcommand{\vto}{\xrightarrow{v}}
    \newcommand{\fidi}{\xrightarrow{\text{fidi}}}

    \newcommand{\eind}{\stackrel{d}{=}}
    \newcommand{\rmd}{\mathrm{d}}

\newcommand{\be}{\begin{equation}}
    \newcommand{\ee}{\end{equation}}

\begin{document}

\title[Maxima of linear processes] 
{Skorokhod $M_{1}$ convergence of maxima of multivariate linear processes with heavy-tailed innovations and random coefficients}

%
\author{Danijel Krizmani\'{c}}

\address{Danijel Krizmani\'{c}\\ Faculty of Mathematics\\
        University of Rijeka\\
        Radmile Matej\v{c}i\'{c} 2, 51000 Rijeka\\
        Croatia}
\email{dkrizmanic@math.uniri.hr}





\begin{abstract}
We derive functional convergence of the partial maxima stochastic processes of multivariate linear processes with weakly dependent heavy-tailed innovations and random coefficients. The convergence takes place in the space of $\mathbb{R}^{d}$--valued c\`{a}dl\`{a}g functions on
$[0,1]$ endowed with the weak Skorokhod $M_{1}$ topology. We also show that this topology in general can not be replaced by the standard (or strong) $M_{1}$ topology.
\end{abstract}

\maketitle

\section{Introduction}
\label{intro}

Let $(X_{i})_{i \in \mathbb{Z}}$ be a strictly stationary sequence of
random variables, and denote by $M_{n} = \max \{X_{1}, X_{2}, \ldots, X_{n}\}$, $n \geq 1$, its partial maxima. The asymptotic distributional behavior of $M_{n}$ is one of the main object of interest of classical extreme value theory. When $(X_{i})$ is an i.i.d.~sequence and there exist constants $a_{n}>0$ and $b_{n}$ such that
\begin{equation}\label{e:EVT}
  \Pr \Big( \frac{M_{n}-b_{n}}{a_{n}} \leq x \Big) \to G(x) \quad \textrm{as} \ n \to \infty,
 \end{equation}
 with non-degenerated limit $G$, then the limit belongs to the class of extreme value distributions (see Gnedenko~\cite{Gn43} and Resnick~\cite{Re87}). In particular, if  the distribution of $X_{1}$ is regularly varying at infinity with some positive index $\alpha$, that is
 $$ \lim_{x \rightarrow \infty} \frac{\Pr(X_{1} >ux)}{\Pr(X_{1}>x)} = u^{-\alpha}$$
 for every $u>0$, then relation (\ref{e:EVT}) holds with $G(x)=\exp \{-x^{-\alpha} \}$, $x>0$ (see Resnick~\cite{Re87}, Proposition 1.11).
 It is known that generalizations of this result to weak convergence of partial maxima processes in the space of c\`{a}dl\`{a}g functions hold. More precisely, relation (\ref{e:EVT}) implies
\begin{equation}\label{e:convPMP}
 a_{n}^{-1}M_{n}(\,\cdot\,)= \bigvee_{i=1}^{\floor {n\,\cdot}} \frac{X_{i}}{a_{n}} \dto Y(\,\cdot\,)
 \end{equation}
in the space $D([0,1], \mathbb{R})$ of real-valued c\`{a}dl\`{a}g functions on
$[0,1]$ endowed with the Skorokhod $J_{1}$ topology, with $Y$
being an extremal process generated by $G$ (see Lamperti~\cite{La64}, and Resnick~\cite{Re87}, Proposition 4.20). Simplifying notation, we sometimes omit brackets and write $a_{n}^{-1}M_{n} \dto Y$.

In the dependent case, Adler~\cite{Ad78} studied $J_{1}$ functional convergence with the weak dependence condition similar to "asymptotic independence" condition introduced by Leadbetter~\cite{Le74}. For stationary sequences of jointly regularly varying random variables Basrak and Tafro~\cite{BaTa16} derived the invariance principle for the partial maximum process in $D([0,1], \mathbb{R})$ with the Skorokhod $M_{1}$ topology. It is well known that weak convergence in relation (\ref{e:convPMP}) holds for a special class of weakly dependent random variables, the linear processes or moving averages processes with i.i.d.~heavy-tailed innovations and deterministic coefficients (see Resnick~\cite{Re87}, Proposition 4.28).

Recently, Krizmani\'{c}~\cite{Kr22-2} showed that the functional convergence in (\ref{e:convPMP}) holds also for linear processes with i.i.d.~innovations and random coefficients. In this paper we aim to generalize this result into two directions, the first one by studying linear processes with weakly dependent innovations (and random coefficients), and the second one by extending this theory to the multivariate setting.
Due to possible clustering of large values, the $J_{1}$ topology becomes inappropriate, and therefore we will use the weaker Skorokhod $M_{1}$ topology. This topology works well if all extremes within each cluster of large values have the same sign. A multivariate version of the convergence in (\ref{e:convPMP}) is well known to hold in the i.i.d.~case (see for example Resnick~\cite{Resnick07}, Proposition 7.2), but also for certain weakly dependent time series, including some $m$--dependent processes and GARCH processes with constant conditional correlations (see Krizmani\'{c}~\cite{Kr17}).

 Say here that the $M_{1}$ convergence in general fails to hold for partial sum processes. Clusters of large values in the sequence $(X_{n})$ may contain positive and negative values yielding the corresponding partial sum processes having jumps of opposite signs within temporal clusters of large values, and this precludes the $M_{1}$ convergence. Under certain conditions one can still obtain convergence in the weaker Skorokhod $M_{2}$ topology (see Krizmani\'{c}~\cite{Kr19}). For partial maxima processes we do not have similar problems with positive and negative values in clusters of large values since these processes are non-decreasing and hence only jumps with positive sign appear in them, which means one can have functional $M_{1}$ convergence.

 The paper is organized as follows. In Section~\ref{S:Pre} we introduce basic notions about linear processes, multivariate regular variation and Skorokhod topologies. In Section~\ref{S:FiniteMA} we derive weak $M_{1}$ convergence of the partial maxima stochastic process for finite order multivariate linear processes with weakly dependent heavy-tailed innovations and random coefficients. The main idea is to transfer the point process convergence obtained by Basrak and Tafro~\cite{BaTa16} to the functional convergence of partial maxima processes using the continuous mapping theorem and some $M_{1}$ continuity properties of an appropriately chosen maximum functional on the space of Radon point measures. In Section~\ref{S:InfiniteMA} we
extend this result to infinite order multivariate linear processes, and show by an example that the convergence in the weak $M_{1}$ topology can not be replaced by the standard $M_{1}$ convergence. Some technical results needed for this extension are given in Appendix.

\section{Preliminaries}\label{S:Pre}

In this section we introduce some basic notions and results on multivariate regular variation, linear processes, Skorokhod topologies and point processes that will be used in the following sections. We use superscripts in parentheses to designate vector components and coordinate functions, for example $a=(a^{(1)}, \ldots, a^{(d)}) \in \mathbb{R}^{d}$ and $x=(x^{(1)}, \ldots, x^{(d)}) \colon [0,1] \to \mathbb{R}^{d}$.
For two vectors $a=(a^{(1)}, \ldots, a^{(d)}), b=(b^{(1)}, \ldots, b^{(d)}) \in \mathbb{R}^{d}$, $a \leq b$ means $a^{(k)} \leq b^{(k)}$ for all $k=1,\ldots, d$. Denote by $(a,b)$ the vector $(a^{(1)}, \ldots, a^{(d)}, b^{(1)}, \ldots, b^{(d)})$. The vector $(a^{(1)}, b^{(1)}, a^{(2)}, b^{(2)}, \ldots, a^{(d)}, b^{(d)})$ will be denoted by $(a^{(i)},b^{(i)})^{*}_{i=1,\ldots,d}$, and the vector $(a^{(1)} \vee b^{(1)}, \ldots, a^{(d)} \vee b^{(d)})$ will be denoted by $a \vee b$, where for $c, d \in \mathbb{R}$ we put $c \vee d = \max \{c, d\}$. Sometimes for convenience we will denote the vector $a$ by $(a^{(i)})_{i=1,\ldots,d}$. For a real number $c$ we write $ca=(ca^{(1)}, \ldots, ca^{(d)})$.

\subsection{Regular variation}

Regular variation on $\mathbb{R}^{d} $ for random vectors is typically formulated in terms of vague convergence on $\EE^{d}= [-\infty, \infty]^{d} \setminus \{ 0\}$. The topology on $\EE^{d}$ is chosen so that a set $B \subseteq \EE^{d}$
has compact closure if and only if it is bounded away from zero,
that is, if there exists $u > 0$ such that $B \subseteq \EE^{d}_u = \{ x
\in \EE^{d} : \| x \| >u \}$. Here $\| \cdot \|$ denotes the max-norm on $\mathbb{R}^{d}$, i.e.
$\displaystyle \| x \|=\max \{ |x^{(i)}| : i=1, \ldots , d\}$ for
$x=(x^{(1)}, \ldots, x^{(d)}) \in \mathbb{R}^{d}$.
Denote by $C_{K}^{+}(\mathbb{E}^{d})$ the class of continuous functions $f \colon \mathbb{E}^{d} \to [0,\infty)$ with compact support.

The $\mathbb{R}^{d}$--valued random vector $\xi$ is (multivariate) regularly varying if there exist $\alpha >0$ and a random vector $\Theta$
on the unit sphere $\mathbb{S}^{d-1} = \{ x \in \mathbb{R}^{d} :
\| x \|=1 \}$ in $\mathbb{R}^{d}$, such that for every $u>0$,
 \begin{equation}\label{e:regvar1}
   \frac{\Pr(\|\xi\| > ux,\,\xi / \| \xi \| \in \cdot \, )}{\Pr(\| \xi \| >x)}
    \wto u^{-\alpha} \Pr( \Theta \in \cdot \,) \qquad \textrm{as} \ x \to \infty,
 \end{equation}
where the arrow ''$\wto$'' denotes the weak convergence of finite measures. This definition does not depend on the choice of the norm, since if $(\ref{e:regvar1})$ holds for some norm on $\mathbb{R}^{d}$, it holds for all norms (of course, with different distributions of $\Theta$). The number $\alpha$ is called the index of regular variation of $\xi$, and the probability measure $\Pr( \Theta \in \cdot \,)$ is called the spectral measure of $\xi$ with respect to the norm $\|\,\cdot\|$.
 Regular variation can be expressed in terms of vague convergence of measures on $\mathcal{B}(\EE^{d})$:
 \begin{equation}\label{e:regvarvague}
 n \Pr ( a_{n}^{-1} \xi \in \cdot\,) \vto \mu (\,\cdot\,),
 \end{equation}
where $(a_{n})$ is a sequence of positive real numbers tending to infinity and $\mu$ is a non-null Radon measure on $\mathcal{B}(\EE^{d})$ with $\mu(\EE^{d} \setminus \mathbb{R}^{d})=0$ and with the following scaling property:
$$ \mu(uB) = u^{-\alpha} \mu(B),  \qquad \forall\,u >0,\,B \in \mathcal{B}(\EE^{d})$$
(Lindskog~\cite{Li04}, Theorems 1.14 and 1.15). We can always choose a sequence $(a_{n})$ such that $n \Pr(\|\xi\| > a_{n}) \to 1$ as $n \to \infty$, for example we can take $a_{n}$ to be the $(1-n^{-1})$--quantile of the distribution function of $\|\xi\|$. This and relation (\ref{e:regvar1}) then imply
\begin{equation}\label{e:nizrv0}
\lim_{n \to \infty} n \Pr(\|\xi\| > u a_{n})=u^{-\alpha} \qquad \textrm{for all} \ u>0.
\end{equation}
In the one-dimensional case regular variation is characterized by
\begin{equation*}\label{e:regvar}
 \Pr(|\xi| > x) = x^{-\alpha} L(x), \qquad x>0,
\end{equation*}
for some slowly varying function $L$ (i.e. $\lim_{x \to \infty}L(tx)/L(x) = 1$ for all $t>0$), and the tail balance condition
\begin{equation*}\label{eq:pq}
  \lim_{x \to \infty} \frac{\Pr(\xi > x)}{\Pr(|\xi| > x)}=p, \qquad
    \lim_{x \to \infty} \frac{\Pr(\xi < -x)}{\Pr(|\xi| > x)}=q,
\end{equation*}
where $p \in [0,1]$ and $p+q=1$.

Now we say that a strictly stationary $\mathbb{R}^{d}$--valued random process $(\xi_{n})_{n \in \mathbb{Z}}$ is jointly regularly varying with index
$\alpha >0$ if for any nonnegative integer $k$ the
$kd$-dimensional random vector $\xi = (\xi_{1}, \ldots ,
\xi_{k})$ is multivariate regularly varying with index $\alpha$.
According to Basrak and Segers~\cite{BaSe} the joint regular variation property of the sequence $(\xi_{n})$ is equivalent to the existence of a process $(Y_n)_{n \in \mathbb{Z}}$
which satisfies $\Pr(\|Y_0\| > y) = y^{-\alpha}$ for $y \geq 1$, and
\begin{equation}\label{e:tailprocess}
  \bigl( (x^{-1}\ \xi_n)_{n \in \mathbb{Z}} \, \big| \, \| \xi_0\| > x \bigr)
  \fidi (Y_n)_{n \in \mathbb{Z}} \qquad \textrm{as} \ x \to \infty,
\end{equation}
where ``$\fidi$'' denotes convergence of finite-dimensional
distributions. The process $(Y_{n})$ is called
the tail process of $(\xi_{n})$. Further, the spectral tail process $(\Theta_{n})_{n \in \mathbb{Z}}$, defined as $\Theta_{n} = Y_{n}/\|Y_{0}\|$, is independent of $\|Y_{0}\|$ and satisfies
\begin{equation*}\label{e:spectraltailprocess}
  \bigl( (\|\xi_{0}\|^{-1}\ \xi_n)_{n \in \mathbb{Z}} \, \big| \, \| \xi_0\| > x \bigr)
  \fidi (\Theta_n)_{n \in \mathbb{Z}} \qquad \textrm{as} \ x \to \infty.
\end{equation*}
The law of $\Theta_{0} \in \mathbb{S}^{d-1}$ is the spectral measure of $\xi_{0}$. The tail process of a time series $(\xi_{n})_{n \in \mathbb{Z}}$, which is jointly regularly varying with index $\alpha >0$ and consisting of independent random vectors, has a very simple representation: $Y_{n}=0$ for $n \neq 0$, and $Y_{0}=Y \Theta_{0}$, where the law of $\Theta_{0}$ is the spectral measure of $\xi_{0}$, $\ln Y$ has exponential distribution with parameter $\alpha$, and $\Theta_{0}$ and $Y$ are independent.

\subsection{Linear processes}

Let $(Z_{i})_{i \in \mathbb{Z}}$ be a strictly stationary sequence of random vectors in $\mathbb{R}^{d}$, and assume $Z_{1}$ is multivariate regularly varying with index $\alpha >0$.
We study multivariate linear processes with random coefficients, defined by
\begin{equation}\label{e:MArandom}
X_{i} = \sum_{j=0}^{\infty}C_{j}Z_{i-j}, \qquad i \in \mathbb{Z},
\end{equation}
where
$(C_{j})_{j \geq 0 }$ is a sequence of
$d \times d$ matrices (with real--valued random variables as entries)
independent of $(Z_{i})$ such that the above series is a.s.~convergent. One sufficient condition for that is
\begin{equation}\label{e:momcond}
\sum_{j=0}^{\infty} \mathrm{E} \|C_{j}\|^{\delta} < \infty \qquad \textrm{for some}  \ \delta < \alpha,\,0 < \delta \leq 1,
\end{equation}
where for a $d \times d$ matrix $C=(C_{i,j})$, $\|C\|$ denotes the operator norm
$$ \|C\| = \sup \{ \|Cx\| : x \in \mathbb{R}^{d}, \|x\| =1 \} = \max_{i=1,\ldots,d} \sum_{j=1}^{d}|C_{i,j}|.$$
The regular variation property and Karamata's theorem imply $\mathrm{E}\|Z_{1}\|^{\beta} < \infty$ for every $\beta \in (0,\alpha)$ (cf.~Bingham et al.~\cite{BiGoTe89}, Proposition 1.5.10), which together with concavity, independence of $C_{i}$ and $Z_{i-j}$ for all $i$ and $j$,
the inequality $\|Cx\| \leq \|C\| \cdot \|x\|$,
and the moment condition (\ref{e:momcond}) yield the a.s.~convergence of the series in (\ref{e:MArandom}):
$$ \mathrm{E}\|X_{i}\|^{\delta} \leq \mathrm{E} \bigg( \sum_{j=0}^{\infty} \|C_{j}Z_{i-j}\| \bigg)^{\delta} \leq \sum_{j=0}^{\infty} \mathrm{E}\|C_{j}\|^{\delta} \mathrm{E}\|Z_{i-j}\|^{\delta} = \mathrm{E}\|Z_{1}\|^{\delta} \sum_{j=0}^{\infty}\mathrm{E}\|C_{j}\|^{\delta} < \infty.$$
When the sequence $(Z_{i})$ is i.i.d.~the a.s.~convergence of the series in the definition of linear processes $X_{i}$ holds also with the following moment conditions:
$\mathrm{E}Z_{1}=0$ if $\alpha >1$, and there exists $\delta \in (0, \alpha)$ such that
$$ \begin{array}{ll}
     \displaystyle \sum_{j=0}^{\infty}\mathrm{E}( \|C_{j}\|^{\alpha - \delta} + \|C_{j}\|^{\alpha + \delta})< \infty & \quad \textrm{if} \ \alpha \in (0,1) \cup (1,2),\\[1.3em]
     \mathrm{E} \bigg[ \bigg( \displaystyle \sum_{j=0}^{\infty}\|C_{j}\|^{\alpha-\delta} \bigg)^{(\alpha+\delta)/(\alpha-\delta)} \bigg] < \infty & \quad \textrm{if} \ \alpha \in \{1, 2 \},\\[1.5em]
      \mathrm{E} \bigg[ \bigg( \displaystyle \sum_{j=0}^{\infty}\|C_{j}\|^{2} \bigg)^{(\alpha+\delta)/2} \bigg] < \infty & \quad \textrm{if} \ \alpha >2,
                                       \end{array}$$
see Kulik and Soulier~\cite{KuSo}, Theorem 4.1.2.

\subsection{Skorokhod topologies}

Denote by $D([0,1], \mathbb{R}^{d})$ the space of all right-continuous $\mathbb{R}^{d}$--valued functions on $[0,1]$ with left limits. For $x \in D([0,1],
\mathbb{R}^{d})$ the completed (thick) graph of $x$ is defined as
\[
  G_{x}
  = \{ (t,z) \in [0,1] \times \mathbb{R}^{d} : z \in [[x(t-), x(t)]]\},
\]
where $x(t-)$ is the left limit of $x$ at $t$ and $[[a,b]]$ is the product segment, i.e.
$[[a,b]]=[a^{(1)},b^{(1)}] \times \ldots \times [a^{(d)},b^{(d)}]$
for $a=(a^{(1)}, \ldots, a^{(d)}), b=(b^{(1)}, \ldots, b^{(d)}) \in
\mathbb{R}^{d}$, and $[a^{(i)}, b^{(i)}]$ coincides with the closed interval $[a^{(i)} \wedge b^{(i)}, a^{(i)} \vee b^{(i)}]$, with $c \wedge d = \min \{c,d \}$ for $c, d \in \mathbb{R}$. On the graph $G_{x}$ we define an order by saying that $(t_{1},z_{1}) \le
(t_{2},z_{2})$ if either
$$ \begin{array}{ll}
   (i) \ \ t_{1} < t_{2}, \quad \textrm{or}& \quad  \\[0.7em]
 (ii) \ \ |x_{j}(t_{1}-) - z_{1}^{(j)}| \le |x_{j}(t_{2}-) - z_{2}^{(j)}| & \quad \textrm{for all} \ j=1, 2, \ldots, d.
                                 \end{array}$$
                                  A weak parametric representation
of the graph $G_{x}$ is a continuous nondecreasing function $(r,u)$
mapping $[0,1]$ into $G_{x}$, with $r$ being the
time component and $u$ being the spatial component, such that $r(0)=0,
r(1)=1$ and $u(1)=x(1)$. Let $\Pi_{w}(x)$ denote the set of weak
parametric representations of the graph $G_{x}$. For $x_{1},x_{2}
\in D([0,1], \mathbb{R}^{d})$ define
\[
  d_{w}(x_{1},x_{2})
  = \inf \{ \|r_{1}-r_{2}\|_{[0,1]} \vee \|u_{1}-u_{2}\|_{[0,1]} : (r_{i},u_{i}) \in \Pi_{w}(x_{i}), i=1,2 \},
\]
where $\|x\|_{[0,1]} = \sup \{ \|x(t)\| : t \in [0,1] \}$. Now we
say that a sequence $(x_{n})_{n}$ converges to $x$ in $D([0,1], \mathbb{R}^{d})$ in the weak Skorokhod $M_{1}$
topology if $d_{w}(x_{n},x)\to 0$ as $n \to \infty$.

If we replace the graph $G_{x}$ with the completed (thin) graph
\[
  \Gamma_{x}
  = \{ (t,z) \in [0,1] \times \mathbb{R}^{d} : z= \lambda x(t-) + (1-\lambda)x(t) \ \text{for some}\ \lambda \in [0,1] \},
\]
and weak parametric representations with strong parametric representations, that is continuous nondecreasing functions $(r,u)$ mapping $[0,1]$ onto $\Gamma_{x}$, then we obtain the standard (or strong) Skorokhod $M_{1}$ topology. This topology is induced by the metric
$$d_{M_{1}}(x_{1},x_{2})
  = \inf \{ \|r_{1}-r_{2}\|_{[0,1]} \vee \|u_{1}-u_{2}\|_{[0,1]} : (r_{i},u_{i}) \in \Pi_{s}(x_{i}), i=1,2 \},$$
where $\Pi_{s}(x)$ is the set of strong parametric representations of the graph $\Gamma_{x}$. Since $\Pi_{s}(x) \subseteq \Pi_{w}(x)$ for all $x \in D([0,1], \mathbb{R}^{d})$, the weak $M_{1}$ topology is weaker than the standard $M_{1}$ topology on $D([0,1],
\mathbb{R}^{d})$, but they coincide for $d=1$.

The weak $M_{1}$ topology coincides with the topology induced by the metric
\begin{equation}\label{e:defdp}
 d_{p}(x_{1},x_{2})=\max \{ d_{M_{1}}(x_{1}^{(j)},x_{2}^{(j)}) :
j=1,\ldots,d\}
\end{equation}
 for $x_{i}=(x_{i}^{(1)}, \ldots, x_{i}^{(d)}) \in D([0,1],
 \mathbb{R}^{d})$ and $i=1,2$. The metric $d_{p}$ induces the product topology on $D([0,1], \mathbb{R}^{d})$.

By using parametric representations in which only the time component $r$ is nondecreasing instead of $(r,u)$ we obtain Skorokhod's weak and strong $M_{2}$ topologies. There is a useful characterization of weak and strong $M_{2}$ topologies by the Hausdorff metric on the space of completed graphs: for the strong version the $M_{2}$ distance between two functions $x_{1}, x_{2} \in D([0,1], \mathbb{R}^{d})$ is given by
$$ d_{M_{2}}(x_{1}, x_{2}) = \bigg(\sup_{a \in \Gamma_{x_{1}}} \inf_{b \in \Gamma_{x_{2}}} d(a,b) \bigg) \vee \bigg(\sup_{a \in \Gamma_{x_{2}}} \inf_{b \in \Gamma_{x_{1}}} d(a,b) \bigg),$$
where $d$ is the metric on $\mathbb{R}^{d+1}$ defined by $d(a, b)=\max\{|a^{(i)}-b^{(i)}| : i=1,\ldots,d+1\}$ for $a=(a^{(1)},\ldots,a^{(d+1)}), b=(b^{(1)}, \ldots, b^{(d+1)}) \in \mathbb{R}^{d+1}$.
The metric $d_{M_{2}}$ induces the strong $M_{2}$ topology, which is weaker than the $M_{1}$ topology. For more details and discussion on the $M_{1}$ and $M_{2}$ topologies we refer to Whitt~\cite{Whitt02}, sections 12.3-12.5.

Since the sample paths of the partial maxima processes in (\ref{e:convPMP}) are non-decreasing, we will restrict our attention to the subspace $D_{\uparrow}([0,1], \mathbb{R}^{d})$ of functions $x$ in $D([0,1], \mathbb{R}^{d})$ for which the coordinate functions $x^{(i)}$ are non-decreasing for all $i=1,\ldots,d$. The following two lemmas about the $M_{1}$ continuity of multiplication and maximum of two $\mathbb{R}$--valued c\`{a}dl\`{a}g functions will be used in the next section. The first lemma is based on Theorem 13.3.2 in Whitt~\cite{Whitt02}, and the second one follows easily from the fact that for monotone functions $M_{1}$ convergence is equivalent to point-wise convergence in a dense subset of $[0,1]$ including $0$ and $1$ (cf.~Corollary 12.5.1 in Whitt~\cite{Whitt02}). Denote by $\textrm{Disc}(x)$ the set of discontinuity points of $x \in D([0,1], \mathbb{R})$.


\begin{lem}\label{l:contmultpl}
Suppose that $x_{n} \to x$ and $y_{n} \to y$ in $D([0,1], \mathbb{R})$ with the $M_{1}$ topology. If for each $t \in \textrm{Disc}(x) \cap \textrm{Disc}(y)$, $x(t)$, $x(t-)$, $y(t)$ and $y(t-)$ are all nonnegative and $[x(t)-x(t-)][y(t)-y(t-)] \geq 0$, then $x_{n}y_{n} \to xy$ in $D([0,1], \mathbb{R})$ with the $M_{1}$ topology, where $(xy)(t) = x(t)y(t)$ for $t \in [0,1]$.
\end{lem}

In particular, if $\textrm{Disc}(x)=\emptyset$, by Lemma~\ref{l:contmultpl} $d_{M_{1}}(x_{n},x) \to 0$ and $d_{M_{1}}(y_{n},y) \to 0$ imply $d_{M_{1}}(x_{n}y_{n}, xy) \to 0$ as $n \to \infty$.

\begin{lem}\label{l:contmax}
The function $h \colon D_{\uparrow}([0,1], \mathbb{R}^{2}) \to D_{\uparrow}([0,1], \mathbb{R})$ defined by
$h(x,y)= x \vee y$, where
$$ (x \vee y)(t) = x(t) \vee y(t), \qquad t \in [0,1],$$
is continuous
when $D_{\uparrow}([0,1], \mathbb{R}^{2})$ is endowed with the weak $M_{1}$ topology and $D_{\uparrow}([0,1], \mathbb{R})$ is endowed with the standard $M_{1}$ topology.
\end{lem}

These two lemmas are easily generalized to the multivariate case, in which we take products or maxima of (some) component functions. We only have to take the weak $M_{1}$ topology for each space $D_{\uparrow}([0,1], \mathbb{R}^{k})$ with $k \geq 2$. For example, using Lemma~\ref{l:contmax} and the definition of the metric $d_{p}$ in (\ref{e:defdp}) (which generates the weak $M_{1}$ topology) it can be shown that the function $h_{1} \colon D_{\uparrow}([0,1], \mathbb{R}^{4}) \to D_{\uparrow}([0,1], \mathbb{R}^{2})$, defined by $h_{1}(x)=(x^{(1)} \vee x^{(2)}, x^{(3)} \vee x^{(4)})$ for $x=(x^{(1)}, x^{(2)}, x^{(3)}, x^{(4)}) \in D_{\uparrow}([0,1], \mathbb{R}^{4})$, is continuous when both spaces $D_{\uparrow}([0,1], \mathbb{R}^{4})$ and $D_{\uparrow}([0,1], \mathbb{R}^{2})$ are endowed with the weak $M_{1}$ topology.
By induction the same holds if we take maxima of more than two components on some places.

The Skorokhod $M$ topologies are weaker than the more frequently used $J_{1}$ topology, but still some useful functions preserve $M$ convergence, which allows to establish new stochastic-process limits from given ones. In Chapter 13 of Whitt~\cite{Whitt02} four functions of this kind are considered: composition, supremum, reflection and inverse. The composition map plays an important role in establishing functional limit theorems involving a random time change, which can, for example, be applied to random sums and maxima (cf.~Theorem 13.2.3 and Corollary 13.3.2 in~\cite{Whitt02}). The supremum and the reflection maps are connected to queueing applications, while the inverse function is used in studying counting processes (see Chapter 14 and Section 7.3 in~\cite{Whitt02} for details).

\subsection{Point processes}\label{ss:PP}

Let $(Z_{i})_{i \in \mathbb{Z}}$ be a strictly stationary sequence of regularly varying $\mathbb{R}^{d}$--valued random vectors with index $\alpha >0$. Assume the elements of this sequence are pairwise asymptotically independent in the sense that
\begin{equation}\label{e:asyind}
\lim_{x \to \infty} \frac{\Pr(\|Z_{i}\| >x, \|Z_{j}\|>x)}{\Pr(\|Z_{1}\|>x)}=0 \qquad \textrm{for all} \ i \neq j.
\end{equation}
We assume also asymptotical independence of the components of each random vector $Z_{i}$:
\begin{equation}\label{e:asyindcomp}
\lim_{x \to \infty} \Pr(|Z_{i}^{(j)}| >x\,\big|\,|Z_{i}^{(k)}|>x)=0 \qquad \textrm{for all} \ j,k \in \{1,\ldots,d\}, j \neq k.
\end{equation}
The asymptotic independence condition (\ref{e:asyind}) implies that the sequence $(Z_{i})$ is jointly regularly varying. This follows by induction from the next result (cf.~Kulik and Soulier~\cite{KuSo}, Proposition 2.1.8).

\begin{prop}\label{p:jrvai}
Let $X$ and $Y$ be identically distributed, regularly varying random vectors with the same index $\alpha >0$. If they are asymptotically independent, then the random vector $(X,Y)$ is regularly varying with index $\alpha$.
\end{prop}
\begin{proof}
Let $(a_{n})$ be a sequence of positive real numbers tending to infinity such that
$$ n \Pr (\|X\| > a_{n}) \to 1 \qquad \textrm{as} \ n \to \infty.$$
 Since $X$ is regularly varying, by (\ref{e:regvarvague}) it holds that $n \Pr(a_{n}^{-1}X \in \cdot\,) \vto \mu(\,\cdot\,)$ as $n \to \infty$, for some Radon measure $\mu$.
For $\epsilon >0$ denote by $B_{\epsilon} = \{ x \in \mathbb{R}^{d} : \|x\| \leq \epsilon \}$ the closed ball of radius $\epsilon$ centered at origin. For $\epsilon_{1}>0$ and $\epsilon_{2}>0$, with $\epsilon = \epsilon_{1} \wedge \epsilon_{2}$, we have
\begin{eqnarray*}
n \Pr(a_{n}^{-1}(X,Y) \in B_{\epsilon_{1}}^{c} \times B_{\epsilon_{2}}^{c}) &=& n \Pr(\|X\| > \epsilon_{1}a_{n}, \|Y\| > \epsilon_{2}a_{n})\\[0.4em]
& \leq & n \Pr(\|X\| > \epsilon a_{n})\,\frac{\Pr(\|X\| > \epsilon a_{n}, \|Y\| > \epsilon a_{n})}{\Pr(\|X\| > \epsilon a_{n})}.
\end{eqnarray*}
Since by (\ref{e:nizrv0}) it holds that $n \Pr(\|X\| > \epsilon a_{n}) \to \epsilon^{-\alpha}$ as $n \to \infty$, an application of the asymptotic independence relation (\ref{e:asyind}) for random vectors $X$ and $Y$ yields
$$ n \Pr(a_{n}^{-1}(X,Y) \in B_{\epsilon_{1}}^{c} \times B_{\epsilon_{2}}^{c}) \to 0 \qquad \textrm{as} \ n \to \infty.$$
Hence $n \Pr(a_{n}^{-1}(X,Y) \in \cdot\,) \vto \widetilde{\mu} (\,\cdot\,)$ on $\mathcal{B}(\EE^{2d} \setminus \{0\})$, where $\widetilde{\mu}$ is a Radon measure that concentrates on $(\{0\} \times \mathbb{R}^{d}) \cup (\mathbb{R}^{d} \times \{0\})$, and for any $B \in \mathcal{B}(\EE^{d})$
$$ \widetilde{\mu}(\{0\} \times B) = \widetilde{\mu}(B \times \{0\})=\mu(B).$$ Hence the random vector $(X,Y)$ is regularly varying. For $r >0$ it holds that
$$ \widetilde{\mu}((r,\infty)^{d} \times \EE^{d}) = \widetilde{\mu}((r,\infty)^{d} \times \{0\}) = \mu((r,\infty)^{d}) = r^{-\alpha}\mu((1,\infty)^{d}).$$
In particular, for $r=1$ we have $\widetilde{\mu}((1,\infty)^{d} \times \EE^{d}) = \mu((1,\infty)^{d})$, and therefore
$$ \widetilde{\mu}((r,\infty)^{d} \times \EE^{d}) = r^{-\alpha}\widetilde{\mu}((1,\infty)^{d} \times \EE^{d}),$$
from which we conclude that the index of regular variation of $(X,Y)$ is $\alpha$.
\end{proof}

Let $(Y_{i})$ be the tail process of the sequence $(Z_{i})$. By the definition of the tail process in (\ref{e:tailprocess}) we have, for $y \geq 1$,
\begin{eqnarray*}
\Pr(\|Y_{0}\| >y) &=& \lim_{x \to \infty} \Pr \bigl( x^{-1} \|Z_{0}\| >y \, \big| \, \| Z_{0} \| > x \bigr) = \lim_{x \to \infty} \frac{\Pr(\|Z_{0}\|> yx, \|Z_{0}\| >x)}{\Pr(\|Z_{0}\|>x)}\\[0.5em]
 & =& \lim_{x \to \infty} \frac{\Pr(\|Z_{0}\|>yx)}{\Pr(\|Z_{0}\|>x)} = y^{-\alpha},
\end{eqnarray*}
 and by (\ref{e:asyind}) we obtain, for every $u \in (0,1)$ and $i \neq 0$,
\begin{eqnarray*}
\Pr(\|Y_{i}\| >u) &=& \lim_{x \to \infty} \Pr \bigl( x^{-1} \|Z_{i}\| >u \, \big| \, \| Z_{0} \| > x \bigr) = \lim_{x \to \infty} \frac{\Pr(\|Z_{i}\|> ux, \|Z_{0}\| >x)}{\Pr(\|Z_{0}\|>x)}\\[0.5em]
 & \leq & \limsup_{x \to \infty} \frac{\Pr(\|Z_{i}\|>ux, \|Z_{0}\|>ux)}{\Pr(\|Z_{0}\|>ux)} \cdot \frac{\Pr(\|Z_{0}\|>ux)}{\Pr(\|Z_{0}\|>x)} = 0,
\end{eqnarray*}
 i.e.~$\Pr(\|Y_{i}\|>u)=0$. Hence we conclude that the tail process is the same as in the i.i.d.~case, that is $Y_{i}=0$ for $i \neq 0$, and $\Pr(\|Y_0\| > y) = y^{-\alpha}$ for $y \geq 1$. Similarly, relation (\ref{e:asyindcomp}) implies that $Y_{0}$ a.s.~has no two nonzero components.

Define the time-space point processes
\begin{equation*}\label{E:ppspacetime}
 N_{n} = \sum_{i=1}^{n} \delta_{(i / n,\,Z_{i} / a_{n})} \qquad \textrm{for all} \ n\in \mathbb{N},
\end{equation*}
with $(a_{n})$ being a sequence of positive real numbers such that
\begin{equation}\label{eq:niz}
n \Pr(\|Z_{1}\| > a_{n}) \to 1 \qquad \textrm{as} \ n \to \infty.
\end{equation}
By relation (\ref{e:nizrv0}) we have
\begin{equation}\label{e:nizrv}
\lim_{n \to \infty} n \Pr(\|Z_{1}\| > u a_{n})=u^{-\alpha} \qquad \textrm{for all} \ u>0,
\end{equation}
and in particular $\Pr(\|Z_{1}\| > u a_{n}) \to 0$ as $n \to \infty$.
The point process convergence for the sequence $(N_{n})$ on the space $[0,1] \times \EE^{d}$ was obtained by Basrak and Tafro~\cite{BaTa16} under joint regular variation and the following two weak dependence conditions.

\begin{cond}\label{c:mixcond1}
There exists a sequence of positive integers $(r_{n})$ such that $r_{n} \to \infty $ and $r_{n} / n \to 0$ as $n \to \infty$ and such that for every nonnegative continuous function $f$ on $[0,1] \times \mathbb{E}^{d}$ with compact support, denoting $k_{n} = \lfloor n / r_{n} \rfloor$, as $n \to \infty$,
\begin{equation}\label{e:mixcon}
 \E \biggl[ \exp \biggl\{ - \sum_{i=1}^{n} f \biggl(\frac{i}{n}, \frac{Z_{i}}{a_{n}}
 \biggr) \biggr\} \biggr]
 - \prod_{k=1}^{k_{n}} \E \biggl[ \exp \biggl\{ - \sum_{i=1}^{r_{n}} f \biggl(\frac{kr_{n}}{n}, \frac{Z_{i}}{a_{n}} \biggr) \biggr\} \biggr] \to 0.
\end{equation}
\end{cond}

\begin{cond}\label{c:mixcond2}
There exists a sequence of positive integers $(r_{n})$ such that $r_{n} \to \infty $ and $r_{n} / n \to 0$ as $n \to \infty$ and such that for every $u > 0$,
\begin{equation}
\label{e:anticluster}
  \lim_{m \to \infty} \limsup_{n \to \infty}
  \Pr \biggl( \max_{m \leq |i| \leq r_{n}} \| Z_{i} \| > ua_{n}\,\bigg|\,\| Z_{0}\|>ua_{n} \biggr) = 0.
\end{equation}
\end{cond}
It can be shown that Condition~\ref{c:mixcond1} holds for strongly mixing random sequences (see Krizmani\'{c}~\cite{Kr10},~\cite{Kr16}). Condition~\ref{c:mixcond2} assures that, roughly speaking, clusters of large values of $\|Z_{i}\|$ have finite mean size, and it is satisfied for $m$--dependent sequences and some other heavy-tailed models, such are stochastic volatility, ARCH and GARCH processes (see Bartkiewicz et al.~\cite{BaJaMiWi09} and Basrak et al.~\cite{BKS}). It also holds under Leadbetter's dependence condition $D'$:
\begin{equation}\label{e:D'cond}
 \lim_{k \to \infty} \limsup_{n \to \infty}~n \sum_{i=1}^{\lfloor n/k \rfloor} \Pr \bigg( \frac{\|Z_{0}\|}{a_{n}} > x, \frac{\|Z_{i}\|}{a_{n}} >x \bigg) = 0 \qquad \textrm{for all} \ x >0,
 \end{equation}
since it implies
\[
    \lim_{n \to \infty} n \sum_{i=1}^{r_{n}} \Pr \bigg( \frac{\|Z_{0}\|}{a_{n}} > u,
    \frac{\|Z_{i}\|}{a_{n}} > u \bigg) = 0 \quad \textrm{for all} \ u
    >0,
\]
for any sequence of positive integers $(r_{n})$ such that $r_{n} \to \infty$ and $r_{n} / n \to 0$ as $n \to \infty$. Note that the asymptotical independence condition (\ref{e:asyind}) holds under condition $D'$.

Therefore, under joint regular variation and Conditions~\ref{c:mixcond1} and~\ref{c:mixcond2}, by Theorem 3.1 in Basrak and Tafro~\cite{BaTa16}, as $n \to \infty$,
\begin{equation}\label{e:BaTa}
N_{n} \dto N = \sum_{i}\sum_{j}\delta_{(T_{i}, P_{i}\eta_{ij})}
\end{equation}
in $[0,1] \times \EE^{d}$, where
\begin{itemize}
  \item[(i)] $\sum_{i=1}^{\infty}\delta_{(T_{i}, P_{i})}$ is a Poisson process on $[0,1] \times (0,\infty)$
with intensity measure $Leb \times \nu$, where $\nu(\rmd x) = \theta \alpha
x^{-\alpha-1}1_{(0,\infty)}(x)\,\rmd x$, and $\theta = \Pr( \sup_{i \leq -1}\|Y_{i}\| \leq 1)$.
  \item[(ii)] $(\sum_{j= 1}^{\infty}\delta_{\eta_{ij}})_{i}$ is an i.i.d.~sequence of point processes in $\EE^{d}$ independent of $\sum_{i}\delta_{(T_{i}, P_{i})}$ and with common distribution equal to the distribution of the point process $\sum_{j}\delta_{\widetilde{Y}_{j}/L(\widetilde{Y})}$, where $L(\widetilde{Y})= \sup_{j \in \mathbb{Z}}\|\widetilde{Y}_{j}\|$ and $\sum_{j}\delta_{\widetilde{Y}_{j}}$ is distributed as $( \sum_{j \in \mathbb{Z}} \delta_{Y_j} \,|\, \sup_{i \le -1} \| Y_i\| \le 1).$
\end{itemize}
Taking into account the form of the tail process $(Y_{i})$ it holds that $\theta=1$ and $N=\sum_{i}\delta_{(T_{i}, P_{i}\eta_{i0})}$ with $\|\eta_{i0}\|=1$.  Hence, denoting $Q_{i}=\eta_{i0}$, the limiting point process in relation (\ref{e:BaTa}) reduces to
\begin{equation}\label{e:BaTa1}
 N = \sum_{i}\delta_{(T_{i}, P_{i}Q_{i})}.
\end{equation}
Since the sequence $(Q_{i})$ is independent of the Poisson process $\sum_{i=1}^{\infty}\delta_{(T_{i}, P_{i})}$, an application of Proposition 5.3 in Resnick~\cite{Resnick07} yields that $\sum_{i}\delta_{(T_{i},P_{i},Q_{i})}$ is a Poisson process on $[0,1] \times (0,\infty) \times \EE^{d}$ with intensity measure $Leb \times \nu \times F$, where $F$ is the common probability distribution of $Q_{i}$. Define $H \colon [0,1] \times (0,\infty) \times \EE^{d} \to [0,1] \times \EE^{d}$ by $H(x,y,z)=(x,yz)$. Then by Proposition 5.2 in Resnick~\cite{Resnick07} we have that $N = \sum_{i}\delta_{H(T_{i}, P_{i}, Q_{i})}$ is a Poisson process with intensity measure $Leb \times \widetilde{\nu}$, where $Leb \times \widetilde{\nu} = (Leb \times \nu \times F) \circ H^{-1}$. For $a=(a^{(1)}, \ldots, a^{(d)}), b=(b^{(1)}, \ldots, b^{(d)}) \in \EE^{d}$ such that $((a,b]]= (a^{(1)}, b^{(1)}] \times \ldots \times (a^{(d)}, b^{(d)}]$ is bounded away from zero, by Fubini's theorem we obtain
\begin{eqnarray*}
\widetilde{\nu}(((a,b]]) &=& (\nu \times F) (\{(y,z) \in (0,\infty) \times \EE^{d} : yz \in ((a,b]]\})\\[0.5em]
  &=&\int_{0}^{\infty} \int_{yz \in ((a,b]]}F(dz) \nu(dy) = \int_{0}^{\infty} \Pr(yQ_{1} \in ((a,b]])\,\alpha y^{-\alpha-1}\,\rmd y.
\end{eqnarray*}

For $x \in \mathbb{R}$ let $x^{+}=|x| 1_{\{x>0\}}$ and $x^{-}=|x| 1_{\{x < 0 \}}$.
Define the maximum functional
$\Phi \colon \mathbf{M}_{p}([0,1] \times \EE^{d}) \to D_{\uparrow}([0,1], \mathbb{R}^{2d^{2}})$
by
\begin{equation}\label{e:maxfuncdef}
 \Phi \Big(\sum_{i}\delta_{(t_{i}, (x_{i}^{(1)},\ldots, x_{i}^{(d)}))} \Big) (t)
  = \bigg( \Big(  \bigvee_{t_{i} \leq t}x_{i}^{(j)+},  \bigvee_{t_{i} \leq t} x_{i}^{(j)-} \Big)^{*}_{j=1,\ldots,d} \bigg)_{k=1,\ldots,d}
 \end{equation}
  for $t \in [0,1]$
(with the convention $\vee \emptyset = 0$), where the space $\mathbf{M}_p([0,1] \times \EE^{d})$ of Radon point
measures on $[0,1] \times \EE^{d}$ is equipped with the vague
topology (see Chapter 3 in Resnick~\cite{Re87}). Note that on the right-hand side in (\ref{e:maxfuncdef}) we repeat the $2d$ coordinates of the vector
$\Big(  \bigvee_{t_{i} \leq t}x_{i}^{(j)+},  \bigvee_{t_{i} \leq t} x_{i}^{(j)-} \Big)^{*}_{j=1,\ldots,d}$ consecutively $d$ times.
Let
\begin{multline*}
  \Lambda =  \{ \eta \in \mathbf{M}_{p}([0,1] \times \EE^{d}) :
   \eta ( \{0,1 \} \times \EE^{d}) = 0 \ \textrm{and} \\[0.2em]
   \eta ([0,1] \times \{ (x^{(1)},\ldots,x^{(d)}) : |x^{(i)}| = \infty \ \textrm{for some} \ i \}) = 0 \}.
\end{multline*}

\begin{lem}\label{l:contfunct}
The maximum functional $\Phi \colon \mathbf{M}_{p}([0,1] \times \EE^{d}) \to D_{\uparrow}([0,1], \mathbb{R}^{2d^{2}}) $ is continuous on the set
$\Lambda$,
when $D_{\uparrow}([0,1], \mathbb{R}^{2d^{2}})$ is endowed with the weak $M_{1}$ topology.
\end{lem}
\begin{proof}
Take an arbitrary $\eta \in \Lambda$ and suppose that $\eta_{n} \vto \eta$ as $n \to \infty$ in $\mathbf{M}_p([0,1] \times
\EE^{d})$. We need to show that
$\Phi(\eta_n) \to \Phi(\eta)$ in $D_{\uparrow}([0,1], \mathbb{R}^{2d^{2}})$ according to the weak $M_1$ topology. By
Theorem 12.5.2 in Whitt~\cite{Whitt02}, it suffices to prove that,
as $n \to \infty$,
$$ d_{p}(\Phi(\eta_{n}), \Phi(\eta)) =
\max_{k=1,\ldots,2d^{2}}d_{M_{1}}(\Phi^{(k)}(\eta_{n}),
\Phi^{(k)}(\eta)) \to 0.$$
By Proposition 3.1 in Krizmani\'{c}~\cite{Kr22-2} it holds that, for every $k=1,\ldots,2d^{2}$,
$$d_{M_{1}}(\Phi^{(k)}(\eta_{n}),
\Phi^{(k)}(\eta)) \to 0 \qquad \textrm{as} \ n \to \infty,$$
and therefore we conclude that $\Phi$ is continuous at $\eta$.
\end{proof}

\section{Finite order linear processes}
\label{S:FiniteMA}

Let $(Z_{i})_{i \in \mathbb{Z}}$ be a strictly stationary sequence of regularly varying $\mathbb{R}^{d}$--valued random vectors with index $\alpha>0$. Fix $m \in \mathbb{N}$, and let
$$ X_{i} = \sum_{j=0}^{m}C_{j}Z_{i-j}, \qquad i \in \mathbb{Z},$$
be a finite order linear process, where $C_{0}, C_{1}, \ldots, C_{m}$ are
random $d \times d$ matrices
 independent of $(Z_{i})$.
Define the corresponding partial maxima process by
\be\label{eq:defWn}
M_{n}(t) = \left\{ \begin{array}{cc}
                                   \displaystyle \frac{1}{a_{n}} \bigvee_{i=1}^{\floor {nt}}X_{i} = \bigg( \frac{1}{a_{n}} \bigvee_{i=1}^{\floor {nt}}X_{i}^{(k)} \bigg)_{k=1,\ldots,d} , & \quad  \displaystyle t \geq \frac{1}{n},\\[1.4em]
                                   \displaystyle \frac{X_{1}}{a_{n}} = \frac{1}{a_{n}}(X_{1}^{(1)}, \ldots, X_{1}^{(d)}) , & \quad \displaystyle  t < \frac{1}{n},
                                 \end{array}\right.
\ee
for $t \in [0,1]$,
with the normalizing sequence $(a_n)$ as in~\eqref{eq:niz}.
For $k,j \in \{1,\ldots,d\}$ let
\begin{equation}\label{e:Cplusminus}
D^{k,j}_{+}=  \bigvee_{i=0}^{m}C_{i;k,j}^{+}  \qquad \textrm{and}
\qquad D^{k,j}_{-}= \bigvee_{i=0}^{m}C_{i;k,j}^{-},
\end{equation}
where $C_{i;k,j}$ is the $(k,j)$th entry of the matrix $C_{i}$, $C_{i;k,j}^{+}=|C_{i;k,j}|1_{\{ C_{i;k,j}>0 \}}$ and $C_{i;k,j}^{-}=|C_{i;k,j}|1_{\{ C_{i;k,j}<0 \}}$.

First we show in the proposition below that a particular maxima process $W_{n}$, constructed from the sequence $(Z_{i})$, converges in $ D_{\uparrow}([0,1], \mathbb{R}^{d})$ with the weak $M_{1}$ topology. Later, in the main result of this section, we will show that the weak $M_{1}$ distance between processes $M_{n}$ and $W_{n}$ is asymptotically negligible (as $n$ tends to infinity), which will imply the functional convergence of the maxima process $M_{n}$. The limiting process will be described in terms of certain extremal processes derived from the point process $N=\sum_{i}\delta_{(T_{i}, P_{i}Q_{i})}$ in relation $(\ref{e:BaTa1})$. Extremal processes can be derived from Poisson processes in the following way. Let $\xi = \sum_{k}\delta_{(t_{k}, j_{k})}$ be a Poisson process on $[0,\infty) \times [0,\infty)^{d}$ with mean measure $Leb \times \mu$, where $\mu$ is a measure on $[0,\infty)^{d}$ satisfying
$$\mu (\{ x \in [0,\infty)^{d} : \|x\| > \delta \}) < \infty$$
for any $\delta >0$. The extremal process $G(\,\cdot\,)$ generated by $\xi$ is defined by
$$ G(t) = \bigvee_{t_{k} \leq t}j_{k}, \qquad t>0.$$
Then for $x \in [0,\infty)^{d}$, $x \neq 0$, and $t>0$ it holds that
$$ \Pr (G(t) \leq x) = e^{-t \mu([[0,x]]^{c})}$$
(cf.~Resnick~\cite{Resnick07}, Section 5.6). The measure $\mu$ is called the exponent measure.

\begin{prop}\label{p:FLT}
Let $(X_{i})$ be a linear process defined by
$$ X_{i} = \sum_{j=0}^{m}C_{j}Z_{i-j}, \qquad i \in \mathbb{Z},$$
where $(Z_{i})_{i \in \mathbb{Z}}$ is a strictly stationary sequence of regularly varying $\mathbb{R}^{d}$--valued random vectors with index $\alpha >0$ that satisfy $(\ref{e:asyind})$ and $(\ref{e:asyindcomp})$, and
 $C_{0}, C_{1}, \ldots, C_{m}$ are random $d\times d$ matrices independent of $(Z_{i})$. Assume Conditions~\ref{c:mixcond1} and~\ref{c:mixcond2} hold with the same sequence $(r_{n})$.
Let $$ W_{n}(t) := \bigg( \bigvee_{i=1}^{\floor {nt}} \bigvee_{j=1}^{d} \frac{1}{a_{n}} \Big(D^{k,j}_{+}Z_{i}^{(j)+} + D^{k,j}_{-}Z_{i}^{(j)-}\Big) \bigg)_{k=1,\ldots,d}, \qquad t \in [0,1],$$
with $D^{k,j}_{+}$ and $D^{k,j}_{-}$ defined in $(\ref{e:Cplusminus})$.
Then, as $n \to \infty$,
\begin{equation}\label{e:pomkonv}
 W_{n}(\,\cdot\,) \dto \bigg( \bigvee_{j=1}^{d} \Big( \widetilde{D}^{k,j}_{+}M^{(j+)}(\,\cdot\,) \vee \widetilde{D}^{k,j}_{-}M^{(j-)}(\,\cdot\,) \Big) \bigg)_{k=1,\ldots,d}
\end{equation}
 in $D_{\uparrow}([0,1], \mathbb{R}^{d})$ with the weak $M_{1}$ topology,
where $M^{(j+)}$ and $M^{(j-)}$ are extremal processes with exponent measures $\nu_{j+}$ and $\nu_{j-}$ respectively, with
$$  \nu_{j+}(\rmd x) = \mathrm{E}(Q_{1}^{(j)+})^{\alpha} \, \alpha x^{-\alpha-1}\, \rmd x
\qquad \textrm{and} \qquad  \nu_{j-}(\rmd x) =  \mathrm{E}(Q_{1}^{(j)-})^{\alpha} \, \alpha x^{-\alpha-1} \, \rmd x$$
for $x>0$ $(j=1,\ldots,d)$,
and $((\widetilde{D}^{k,j}_{+}, \widetilde{D}^{k,j}_{-})^{*}_{j =1, \ldots, d})_{k=1,\ldots,d}$ is a $2d^{2}$--dimensional random vector, independent of $(M^{(j+)}, M^{(j-)})_{j=1,\ldots,d}$, such that
$$((\widetilde{D}^{k,j}_{+}, \widetilde{D}^{k,j}_{-})^{*}_{j =1, \ldots, d})_{k=1,\ldots,d} \eind ((D^{k,j}_{+}, D^{k,j}_{-})^{*}_{j =1, \ldots, d})_{k=1,\ldots,d}.$$
\end{prop}

\begin{rem}
In Proposition~\ref{p:FLT}, as well as in the sequel of this paper, we suppose $M^{(j+)}$ is an extremal process if $\mathrm{E}(Q_{1}^{(j)+})^{\alpha}>0$, and a zero process if this quantity is equal to zero. Analogously for $M^{(j-)}$.
\end{rem}

\begin{proof}[Proof of Proposition~\ref{p:FLT}]
By induction form Proposition~\ref{p:jrvai} it follows that the sequence $(Z_{i})$ is jointly regularly varying with index $\alpha$. This, with Conditions~\ref{c:mixcond1} and~\ref{c:mixcond2}, imply the point process convergence in (\ref{e:BaTa}) with the limiting point process $N$ described in (\ref{e:BaTa1}). Since $N$ is a Poisson process, it almost surely belongs to the set $\Lambda$. Therefore, since by Lemma~\ref{l:contfunct} the maximum functional $\Phi$ is continuous on $\Lambda$, the continuous mapping theorem (see for instance Theorem 3.1 in Resnick~\cite{Resnick07}) applied to the convergence in (\ref{e:BaTa}) yields $\Phi(N_{n}) \dto \Phi(N)$ in $D_{\uparrow}([0,1], \mathbb{R}^{2d^{2}})$ under the weak $M_{1}$ topology, i.e.
\begin{eqnarray}\label{e:conv11}
 \nonumber W_{n}^{\star}(\,\cdot\,)  :=  \bigg( \Big(  \bigvee_{i=1}^{\floor{n\,\cdot}}\frac{Z_{i}^{(j)+}}{a_{n}},  \bigvee_{i=1}^{\floor{n\,\cdot} } \frac{Z_{i}^{(j)-}}{a_{n}} \Big)^{*}_{j=1,\ldots,d} \bigg)_{k=1,\ldots,d}&&\\[0.6em]
  & \hspace*{-40em} \dto & \hspace*{-20em} W(\,\cdot\,) :=
 \bigg( \Big(  \bigvee_{T_{i} \leq\,\cdot}P_{i}Q_{i}^{(j)+},  \bigvee_{T_{i} \leq\,\cdot} P_{i}Q_{i}^{(j)-} \Big)^{*}_{j=1,\ldots,d} \bigg)_{k=1,\ldots,d}
\end{eqnarray}
in $D_{\uparrow}([0,1], \mathbb{R}^{2d^{2}})$ under the weak $M_{1}$ topology.

The space $D([0,1],\mathbb{R})$ equipped with the Skorokhod $J_{1}$ topology is a Polish space, i.e.~metrizable as a complete separable metric space (see Billingsley~\cite{Bi68}, Section 14), and thus the same holds for the $M_{1}$ topology, since it is topologically complete (see Whitt~\cite{Whitt02}, Section 12.8) and separability remains preserved in the weaker topology. The space $D_{\uparrow}([0,1], \mathbb{R})$ is a closed subspace of $D([0,1], \mathbb{R})$ (cf.~Lemma 13.2.3 in Whitt~\cite{Whitt02}), and hence also Polish. The space $D_{\uparrow}([0,1], \mathbb{R}^{2d^{2}})$ equipped with the weak $M_{1}$ topology is separable as a direct product of $2d^{2}$ separable topological spaces, and also topologically complete since the product metric in (\ref{e:defdp}) inherits the completeness of the component metrics. Therefore $D_{\uparrow}([0,1], \mathbb{R}^{2d^{2}})$ with the weak $M_{1}$ topology is also a Polish space,
and hence by Corollary 5.18 in Kallenberg~\cite{Ka97} we can find a random vector $((\widetilde{D}^{k,j}_{+}, \widetilde{D}^{k,j}_{-})^{*}_{j =1, \ldots, d})_{k=1,\ldots,d}$, independent of $W$, such that
\begin{equation}\label{e:eqdistrC}
((\widetilde{D}^{k,j}_{+}, \widetilde{D}^{k,j}_{-})^{*}_{j =1, \ldots, d})_{k=1,\ldots,d} \eind ((D^{k,j}_{+}, D^{k,j}_{-})^{*}_{j =1, \ldots, d})_{k=1,\ldots,d}.
\end{equation}
This, relation (\ref{e:conv11}) and the fact that $((D^{k,j}_{+}, D^{k,j}_{-})^{*}_{j =1, \ldots, d})_{k=1,\ldots,d}$ is independent of $W_{n}^{\star}$, by an application of Theorem 3.29 in Kallenberg~\cite{Ka97}, imply that
  \begin{equation}\label{e:zajedkonvK}
   (B, W_{n}^{\star}) \dto (\widetilde{B}, W) \qquad \textrm{as} \ n \to \infty,
  \end{equation}
  in $D_{\uparrow}([0,1], \mathbb{R}^{4d^{2}})$ with the product $M_{1}$ topology, where $B= ((B^{k,j}_{+}, B^{k,j}_{-})^{*}_{j =1, \ldots, d})_{k=1,\ldots,d}$ and $\widetilde{B} = ((\widetilde{B}^{k,j}_{+}, \widetilde{B}^{k,j}_{-})^{*}_{j =1, \ldots, d})_{k=1,\ldots,d}$ are random elements in $D_{\uparrow}([0,1], \mathbb{R}^{2d^{2}})$ such that $B^{k,j}_{+}(t)=D^{k,j}_{+}$, $B^{k,j}_{-}(t)=D^{k,j}_{-}$, $\widetilde{B}^{k,j}_{+}(t)=\widetilde{D}^{k,j}_{+}$ and $\widetilde{B}^{k,j}_{-}(t)=\widetilde{D}^{k,j}_{+}$ for $t \in [0,1]$. Note that the components of $B$ and $\widetilde{B}$ have constant sample paths.

Let $g \colon D_{\uparrow}([0,1], \mathbb{R}^{4d^{2}}) \to D_{\uparrow}([0,1], \mathbb{R}^{2d^{2}})$ be a function defined by
$$ g(x) = (x^{(1)}x^{(2d^{2}+1)}, x^{(2)}x^{(2d^{2}+2)}, \ldots, x^{(2d^{2})}x^{(4d^{2})})$$
 for $x=(x^{(1)},\ldots, x^{(4d^{2})}) \in D_{\uparrow}([0,1], \mathbb{R}^{4d^{2}})$.
Basically, for $x \in D_{\uparrow}([0,1], \mathbb{R}^{4d^{2}})$, $g(x)$ is an element in $D_{\uparrow}([0,1], \mathbb{R}^{2d^{2}})$ whose $k$--th component is the product of the $k$--th and $(2d^{2}+k)$th components of $x$ ($k=1,\ldots, 2d^{2})$.
 Denote by $A$ the set of all functions in $D_{\uparrow}([0,1], \mathbb{R}^{4d^{2}})$ for which the first $2d^{2}$ component functions have no discontinuity points, i.e.
$$ A = \{ x=(x^{(i)})_{i=1,\ldots,4d^{2}} \in D_{\uparrow}([0,1], \mathbb{R}^{4d^{2}}) : \textrm{Disc}(x^{(i)}) = \emptyset \ \textrm{for all} \ i=1,\ldots, 2d^{2} \}.$$
Lemma~\ref{l:contmultpl} implies the function $g$ is continuous on the set $A$ in the weak $M_{1}$ topology, and hence $\textrm{Disc}(g) \subseteq A^{c}$. Denoting
$A_{1} = \{ x \in D_{\uparrow}([0,1], \mathbb{R}) : \textrm{Disc}(x)= \emptyset \}$ we obtain
\begin{eqnarray*}
\Pr[ (\widetilde{B}, W) \in \textrm{Disc}(g) ] & \leq & \Pr [ (\widetilde{B}, W)  \in A^{c}] \leq \Pr \bigg[ \bigcup_{i=1}^{2d^{2}} \{ \widetilde{B}^{(i)} \in A_{1}^{c} \} \bigg]\\[0.1em]
 & \leq & \sum_{j,k=1}^{d} \Big( \Pr( \widetilde{B}^{k,j}_{+} \in A_{1}^{c}) + \Pr( \widetilde{B}^{k,j}_{-} \in A_{1}^{c}) \Big) = 0,
 \end{eqnarray*}
where the last equality holds since $\widetilde{B}^{k,j}_{+}$ and $\widetilde{B}^{k,j}_{-}$ have no discontinuity points. Therefore we can apply the continuous mapping theorem to the convergence in relation (\ref{e:zajedkonvK}), and this yields
$g(B, W_{n}^{\star}) \dto g(\widetilde{B}, W)$ as $n \to \infty$, i.e.
\begin{eqnarray}\label{e:zajedkonvK2}
  \nonumber \widetilde{W_{n}^{\star}}(\,\cdot\,)  := \bigg( \Big(  \bigvee_{i=1}^{\floor{n\,\cdot}}\frac{D^{k,j}_{+}Z_{j}^{(j)+}}{a_{n}},  \bigvee_{i=1}^{\floor{n\,\cdot}} \frac{D^{k,j}_{-}Z_{i}^{(j)-}}{a_{n}} \Big)^{*}_{j=1,\ldots,d} \bigg)_{k=1,\ldots,d}&& \\[0.5em]
 &   \hspace*{-51em} \dto & \hspace*{-25.5em} \widetilde{W} (\,\cdot\,) := \bigg( \Big(  \bigvee_{T_{i} \leq\,\cdot} \widetilde{D}^{k,j}_{+}P_{i}Q_{i}^{(j)+},  \bigvee_{T_{i} \leq\,\cdot} \widetilde{D}^{k,j}_{-}P_{i}Q_{i}^{(j)-} \Big)^{*}_{j=1,\ldots,d} \bigg)_{k=1,\ldots,d}
\end{eqnarray}
 in $D_{\uparrow}([0,1], \mathbb{R}^{2d^{2}})$ with the weak $M_{1}$ topology.
 Let $h \colon D_{\uparrow}([0,1], \mathbb{R}^{2d^{2}}) \to D_{\uparrow}([0,1], \mathbb{R}^{d})$ be a function defined by
$$ h(x) = \Big( \bigvee_{i=1}^{2d}x^{(i)}, \bigvee_{i=2d+1}^{4d}x^{(i)}, \ldots, \bigvee_{i=2(d-1)d+1}^{2d^{2}}x^{(i)} \Big), \quad x=(x^{(i)})_{i=1,\ldots,2d^{2}} \in D_{\uparrow}([0,1], \mathbb{R}^{2d^{2}}).$$
A multivariate version of Lemma~\ref{l:contmax} shows that $h$ is continuous when both spaces $D_{\uparrow}([0,1], \mathbb{R}^{2d^{2}})$ and $D_{\uparrow}([0,1], \mathbb{R}^{d})$ are endowed with the weak $M_{1}$ topology. Hence an application of the continuous mapping theorem to
the convergence in relation (\ref{e:zajedkonvK2}) yields
 $ h( \widetilde{W_{n}^{\star}}) \dto h(\widetilde{W})$ as $n \to \infty$, i.e.
 $$ \bigg( \bigvee_{i=1}^{\floor {n\,\cdot}} \bigvee_{j=1}^{d} \frac{D^{k,j}_{+}Z_{i}^{(j)+} \vee D^{k,j}_{-}Z_{i}^{(j)-}}{a_{n}} \bigg)_{k=1,\ldots,d}
 \dto \bigg( \bigvee_{T_{i} \leq\,\cdot}^{} \bigvee_{j=1}^{d} (\widetilde{D}^{k,j}_{+}P_{i}Q_{i}^{(j)+}\,\vee\,\widetilde{D}^{k,j}_{-}P_{i}Q_{i}^{(j)-}) \bigg)_{k=1,\ldots,d}$$
  in $D_{\uparrow}([0,1], \mathbb{R}^{d})$ with the weak $M_{1}$ topology. Note that $h( \widetilde{W_{n}^{\star}})$ is equal to $W_{n}$.
 To finish it remains to show that $h(\widetilde{W})$ is equal to the limiting process in relation (\ref{e:pomkonv}). By an application of Propositions 5.2 and 5.3 in Resnick~\cite{Resnick07} we obtain that for every $j=1,\ldots,d$ the point processes $\sum_{i}\delta_{(T_{i}, P_{i}Q_{i}^{(j)+})}$ and $\sum_{i}\delta_{(T_{i}, P_{i}Q_{i}^{(j)-})}$ are Poisson processes with intensity measures $\emph{Leb} \times \nu_{j+}$ and $\emph{Leb} \times \nu_{j-}$ respectively, where
$$  \nu_{j+}(\rmd x) = \mathrm{E}(Q_{1}^{(j)+})^{\alpha} \, \alpha x^{-\alpha-1}\, \rmd x
\qquad \textrm{and} \qquad  \nu_{j-}(\rmd x) =  \mathrm{E}(Q_{1}^{(j)-})^{\alpha} \, \alpha x^{-\alpha-1} \, \rmd x$$
for $x>0$. From this we conclude that the processes
$$ M^{(j+)}(\,\cdot\,) :=  \bigvee_{T_{i} \leq\,\cdot} P_{i}Q_{i}^{(j)+} \qquad \textrm{and} \qquad M^{(j-)}(\,\cdot\,) :=\bigvee_{T_{i} \leq\,\cdot} P_{i}Q_{i}^{(j)-}$$
are extremal processes with exponent measures $\nu_{j+}$ and $\nu_{j-}$ respectively (see Resnick~\cite{Re87}, Section 4.3; Resnick~\cite{Resnick07}, p.~161), and hence
$$ h(\widetilde{W}) = \bigg( \bigvee_{j=1}^{d} \Big(\widetilde{D}^{k,j}_{+}M^{(j+)} \vee \widetilde{D}^{k,j}_{-}M^{(j-)} \Big) \bigg)_{k=1,\ldots,d}, \qquad t \in [0,1].$$
\end{proof}

Denote by $M$ the limiting process in relation (\ref{e:pomkonv}), i.e.
\begin{equation}\label{e:limprocess}
 M(t) = \bigg( \bigvee_{j=1}^{d} \Big(\widetilde{D}^{k,j}_{+}M^{(j+)}(t) \vee \widetilde{D}^{k,j}_{-}M^{(j-)}(t) \Big) \bigg)_{k=1,\ldots,d}, \qquad t \in [0,1],
\end{equation}
where $M^{(j+)}$ and $M^{(j-)}$ are extremal processes with exponent measures $\nu_{j+}$ and $\nu_{j-}$ respectively, given by
$$  \nu_{j+}(\rmd x) = \mathrm{E}(Q_{1}^{(j)+})^{\alpha} \, \alpha x^{-\alpha-1}\, \rmd x
\qquad \textrm{and} \qquad  \nu_{j-}(\rmd x) =  \mathrm{E}(Q_{1}^{(j)-})^{\alpha} \, \alpha x^{-\alpha-1} \, \rmd x$$
for $x>0$ $(j=1,\ldots,d)$, with $\sum_{i}\delta_{(T_{i},P_{i}Q_{i})}$ being the Poisson process from (\ref{e:BaTa1}),
and $((\widetilde{D}^{k,j}_{+}, \widetilde{D}^{k,j}_{-})^{*}_{j =1, \ldots, d})_{k=1,\ldots,d}$ is a $2d^{2}$--dimensional random vector, independent of $(M^{(j+)}, M^{(j-)})_{j=1,\ldots,d}$, such that
$$((\widetilde{D}^{k,j}_{+}, \widetilde{D}^{k,j}_{-})^{*}_{j =1, \ldots, d})_{k=1,\ldots,d} \eind ((D^{k,j}_{+}, D^{k,j}_{-})^{*}_{j =1, \ldots, d})_{k=1,\ldots,d}.$$
Taking into account the proof of Proposition~\ref{p:FLT} observe that
$$M^{(j+)}(t) :=  \bigvee_{T_{i} \leq t} P_{i}Q_{i}^{(j)+} \quad \textrm{and} \quad M^{(j-)}(t) :=\bigvee_{T_{i} \leq t} P_{i}Q_{i}^{(j)-}, \qquad t \in [0,1].$$

The proof of the next theorem relies on the proof of Theorem 3.3 in Krizmani\'{c}~\cite{Kr22-2} where the functional convergence of the partial maxima process is established for univariate linear processes with i.i.d.~innovations and random coefficients. We will omit some details of those parts of the proof that remain the same in our case, but we will carefully write down those parts that differ due to the multivariate setting and weak dependence of innovations.

\begin{thm}\label{t:FLT}
 Let $(Z_{i})_{i \in \mathbb{Z}}$ be a strictly stationary sequence of regularly varying $\mathbb{R}^{d}$--valued random vectors with index $\alpha >0$ that satisfy $(\ref{e:asyind})$ and $(\ref{e:asyindcomp})$, and let
 $C_{0}, C_{1}, \ldots, C_{m}$ be random $d\times d$ matrices independent of $(Z_{i})$. Assume Conditions~\ref{c:mixcond1} and~\ref{c:mixcond2} hold with the same sequence $(r_{n})$.
Then, as $n \to \infty$,
$$ M_{n}(\,\cdot\,) \dto  M(\,\cdot\,)$$
in $D_{\uparrow}([0,1], \mathbb{R}^{d})$ endowed with the weak $M_{1}$ topology.
\end{thm}
\begin{proof}
Let $W_{n}$ be as defined in Proposition~\ref{p:FLT}.
If we show that for every $\delta >0$,
$$ \lim_{n \to \infty} \Pr[d_{p}(W_{n}, M_{n}) > \delta]=0,$$
then from Proposition~\ref{p:FLT} by an application of Slutsky's theorem (see for instance Theorem 3.4 in Resnick~\cite{Resnick07}) it will follow $M_{n} \dto M$ as $n \to \infty$
in $D_{\uparrow}([0,1], \mathbb{R}^{d})$ with the weak $M_{1}$ topology.
Since
$$ \Pr[d_{p}(W_{n}, M_{n}) > \delta] \leq \sum_{j=1}^{d} \Pr [d_{M_{1}}(W_{n}^{(j)}, M_{n}^{(j)}) > \delta],$$
we need to show
$$ \lim_{n \to \infty} \Pr[d_{M_{1}}(W_{n}^{(j)}, M_{n}^{(j)}) > \delta]=0,$$
for every $j=1,\ldots,d$, but it is enough to prove the last relation only for $j=1$ (since the proof is analogous for all coordinates $j$). In fact
it suffices to show that
\begin{equation}\label{e:max1}
\lim_{n \to \infty} \Pr[d_{M_{2}}(W_{n}^{(1)}, M_{n}^{(1)}) > \delta]=0,
\end{equation}
since for $x, y \in D_{\uparrow}([0,1], \mathbb{R})$ it holds that $d_{M_{2}}(x, y)= d_{M_{1}}^{*}(x, y)$, where $d_{M_{1}}^{*}$ is a complete metric topologically equivalent to $d_{M_{1}}$ (see Remark 12.8.1 in Whitt~\cite{Whitt02}). Generally,
$$ d_{M_{1}}^{*}(x, y) = d_{M_{2}}(x, y) + \lambda (\widehat{\omega}(x,\cdot), \widehat{\omega}(y,\cdot)),$$
where $\lambda$ is the L\'{e}vy metric on a space of distributions
$$ \lambda (F_{1},F_{2}) = \inf \{ \epsilon >0 : F_{2}(z-\epsilon) - \epsilon \leq F_{1}(z) \leq F_{2}(z+\epsilon) + \epsilon \ \ \textrm{for all} \ z \}$$
and
$$ \widehat{\omega}(z,s) = \left\{ \begin{array}{cc}
                                   \omega(z,e^{s}), & \quad s<0,\\[0.4em]
                                   \omega(z,1), & \quad s \geq 0,
                                 \end{array}\right.$$
with
$ \omega (z,\rho) = \sup_{0 \leq t \leq 1} \omega (z, t, \rho)$ and
\begin{equation*}\label{e:oscillationf}
\omega(z,t,\rho) = \sup_{0 \vee (t-\rho) \leq t_{1} < t_{2} < t_{3} \leq (t+\rho) \wedge 1} ||z(t_{2}) - [z(t_{1}), z(t_{3})]||
\end{equation*}
where $\rho>0$ and $\|c-A\|$ denotes the distance between a point $c$ and a subset $A \subseteq \mathbb{R}$.
Since the processes $W_{n}^{(1)}$ and $M_{n}^{(1)}$ are nondecreasing, for $t_{1} < t_{2} < t_{3}$ it holds that
$ \|W_{n}^{(1)}(t_{2}) - [W_{n}^{(1)}(t_{1}), W_{n}^{(1)}(t_{3})] \|=0$, which implies $\omega(W_{n}^{(1)}, \rho)=0$ for all $\rho>0$, and similarly $\omega(M_{n}^{(1)}, \rho)=0$. Hence $\lambda (W_{n}^{(1)}, M_{n}^{(1)})=0$, and $d_{M_{1}}^{*}(W_{n}^{(1)}, M_{n}^{(1)}) = d_{M_{2}}(W_{n}^{(1)}, M_{n}^{(1)})$.

In order to show (\ref{e:max1}) fix $\delta >0$ and let $n \in \mathbb{N}$ be large enough, i.e.~$n > \max\{2m, 2m/\delta \}$.
By the definition of the metric $d_{M_{2}}$ we have
\begin{equation*}
  d_{M_{2}}(W_{n}^{(1)}, M_{n}^{(1)}) = \bigg(\sup_{v \in \Gamma_{W_{n}^{(1)}}} \inf_{z \in \Gamma_{ M_{n}^{(1)}}} d(v,z) \bigg) \vee \bigg(\sup_{v \in \Gamma_{ M_{n}^{(1)}}} \inf_{z \in \Gamma_{W_{n}^{(1)}}} d(v,z) \bigg)  = : R_{n} \vee T_{n}.
\end{equation*}
Hence
\be\label{eq:AB}
\Pr [d_{M_{2}}(W_{n}^{(1)}, M_{n}^{(1)})> \delta ] \leq \Pr(R_{n}>\delta) + \Pr(T_{n}>\delta).
\ee
In order to estimate the first term on the right-hand side of (\ref{eq:AB}) define
$$ D_{n} = \{\exists\,v \in \Gamma_{W_{n}^{(1)}} \ \textrm{such that} \ d(v,z) > \delta \ \textrm{for every} \ z \in \Gamma_{ M_{n}^{(1)}} \},$$
and note that by the definition of $R_{n}$
\begin{equation}\label{e:Yn}
 \{R_{n} > \delta\} \subseteq D_{n}.
\end{equation}
On the event $D_{n}$ it holds that $d(v, \Gamma_{M_{n}^{(1)}})> \delta$. Let $v=(t_{v},x_{v})$. Then as in the proof of Theorem 3.3 in Krizmani\'{c}~\cite{Kr22-2} it holds that
\begin{equation}\label{e:i*}
 \Big| W_{n}^{(1)} \Big( \frac{i^{*}}{n} \Big) - M_{n}^{(1)} \Big( \frac{i^{*}}{n} \Big) \Big| \geq d(v, \Gamma_{M_{n}^{(1)}})> \delta,
\end{equation}
with $i^{*}=\floor{nt_{v}}$ or $i^{*}=\floor{nt_{v}}-1$, where the former happens when $v$ lies on a horizontal part of the completed graph $\Gamma_{W_{n}^{(1)}}$ or on a vertical part of the graph with $x_{v} > M_{n}^{(1)}((\floor{nt_{v}}-1)/n)$, and the latter when $v$ lies on a vertical part of the graph with $x_{v} < M_{n}^{(1)}((\floor{nt_{v}}-1)/n)$. Since $|i^{*}/n -(i^{*}+l)/n| \leq m/n < \delta$ for every $l=1,\ldots,m$ (such that $i^{*}+l \leq n$), from the definition of the set $D_{n}$ one can similarly conclude that
\begin{equation}\label{e:i*q}
 \Big| W_{n}^{(1)} \Big( \frac{i^{*}}{n} \Big) - M_{n}^{(1)} \Big( \frac{i^{*}+l}{n} \Big) \Big| > \delta.
\end{equation}
Let $D = \bigvee_{k,j=1,\ldots,d}(D^{k,j}_{+} \vee D^{k,j}_{-})$. This implies $|C_{i:k,j}| \leq D$ for all $i \in \{0,\ldots,m\}$ and $k,j \in \{1,\ldots,d\}$.
We claim that
\begin{equation}\label{e:estim1}
 D_{n} \subseteq H_{n, 1} \cup H_{n, 2} \cup H_{n, 3} \cup H_{n, 4},
\end{equation}
where
\begin{eqnarray}
 \nonumber H_{n, 1} & = & \bigg\{ \exists\,l \in \{-m,\ldots,m\} \cup \{n-m+1, \ldots, n\} \ \textrm{such that} \ \frac{ D\|Z_{l}\|}{a_{n}} > \frac{\delta}{8(m+1)d} \bigg\},\\[0.1em]
  \nonumber H_{n, 2} & = & \bigg\{ \exists\,k \in \{1, \ldots, n\} \ \textrm{and} \ \exists\,l \in \{k-m,\ldots,k+m\} \setminus \{k\} \ \textrm{such that}\\[0.1em]
  \nonumber & & \ \frac{ D\|Z_{k}\|}{a_{n}} > \frac{\delta}{8(m+1)d} \ \textrm{and} \  \frac{ D\|Z_{l}\|}{a_{n}} > \frac{\delta}{8(m+1)d}  \bigg\},\\[0.1em]
   \nonumber H_{n, 3} & = & \bigg\{ \exists\,k \in \{1, \ldots, n\}, \ \exists\,j_{0} \in \{1, \ldots, d \} \ \textrm{and} \ \exists\,p \in \{1,\ldots,d\} \setminus \{j_{0}\} \ \textrm{such that}\\[0.1em]
  \nonumber & & \ \frac{ D|Z_{k}^{(j_{0})}|}{a_{n}} > \frac{\delta}{8(m+1)d} \ \textrm{and} \  \frac{ D|Z_{k}^{(p)}|}{a_{n}} > \frac{\delta}{8(m+1)d}  \bigg\},\\[0.1em]
  \nonumber H_{n, 4} & = & \bigg\{ \exists\,k \in \{1, \ldots, n\}, \ \exists\,j \in \{1,\ldots,n\} \setminus \{k,\ldots,k+m\}, \ \exists\,l_{1}, l_{2} \in \{0,\ldots,m\}\\[0.1em]
  \nonumber & & \ \textrm{and} \ \exists\,s_{1}, s_{2} \in \{1,\ldots,d\} \ \textrm{such that} \ (l_{1},s_{1}) \neq (l_{2}, s_{2}), \  \frac{ D\|Z_{k}\|}{a_{n}} > \frac{\delta}{8(m+1)d},\\[0.1em]
  \nonumber & & \ \frac{ D|Z_{j-l_{1}}^{(s_{1})}|}{a_{n}} > \frac{\delta}{8(m+1)d} \ \textrm{and} \  \frac{ D|Z_{j-l_{2}}^{(s_{2})}|}{a_{n}} > \frac{\delta}{8(m+1)d}  \bigg\}.
\end{eqnarray}
Note that relation (\ref{e:estim1}) will be proven if we show that
$$ \widehat{D}_{n} := D_{n} \cap (H_{n, 1} \cup H_{n, 2} \cup H_{n, 3})^{c} \subseteq H_{n, 4}.$$
Assume the event $\widehat{D}_{n}$ occurs. Then necessarily $W_{n}^{(1)}(i^{*}/n) > \delta / [8(m+1)d]$. Indeed, if $W_{n}^{(1)}(i^{*}/n) \leq \delta / [8(m+1)d]$, that is
$$ \bigvee_{i=1}^{i^{*}} \bigvee_{j=1}^{d} \frac{1}{a_{n}} \Big(D^{1,j}_{+}Z_{i}^{(j)+} \vee D^{1,j}_{-}Z_{i}^{(j)-}\Big)  = W_{n}^{(1)} \Big( \frac{i^{*}}{n} \Big) \leq \frac{\delta}{8(m+1)d},$$
then for every $s \in \{ m+1, \ldots, i^{*}\}$ it holds that
\begin{eqnarray}\label{e:phipm}
   \nonumber \frac{X_{s}^{(1)}}{a_{n}}  &=&  \sum_{r=0}^{m} \sum_{j=1}^{d} \frac{C_{r;1,j} Z_{s-r}^{(j)}}{a_{n}} \leq  \sum_{r=0}^{m} \sum_{j=1}^{d}  \frac{D^{1,j}_{+} Z_{s-r}^{(j)+} \vee D^{1,j}_{-} Z_{s-r}^{(j)-} }{a_{n}}\\[0.4em]
    &\leq& \sum_{r=0}^{m} \sum_{j=1}^{d} \frac{\delta}{8(m+1)d } = \frac{\delta}{8},
\end{eqnarray}
since by the definition of $D^{1,j}_{+}$ and $D^{1,j}_{-}$ we have $D^{1,j}_{+} Z_{s-r}^{(j)+} \geq 0$, $D^{1,j}_{-} Z_{s-r}^{(j)-} \geq 0$ and
$$ C_{r;1,j} Z_{s-r}^{(j)} \leq \left\{ \begin{array}{cl}
                                   D^{1,j}_{+} Z_{s-r}^{(j)+}, & \quad \textrm{if} \ C_{r;1,j} >0 \ \textrm{and} \ Z_{s-r}^{(j)}>0,\\[0.5em]
                                   D^{1,j}_{-} Z_{s-r}^{(j)-}, & \quad \textrm{if} \ C_{r;1,j} <0 \ \textrm{and} \ Z_{s-r}^{(j)}<0 ,\\[0.5em]
                                   0, & \quad \textrm{if} \  C_{r;1,j} \cdot Z_{s-r}^{(j)} \leq 0.
                                 \end{array}\right.$$
Since the event $H_{n, 1}^{c}$ occurs, for every $s \in \{1, \ldots, m\}$ we also have
\begin{equation}\label{e:phipm2}
  \frac{|X_{s}^{(1)}|}{a_{n}}  \leq  \sum_{r=0}^{m} \sum_{j=1}^{d} |C_{r;1,j}| \frac{|Z_{s-r}^{(j)}|}{a_{n}} \leq \sum_{r=0}^{m} \sum_{j=1}^{d}  \frac{D \|Z_{s-r}\|}{a_{n}} \leq (m+1)d\,\frac{\delta}{8(m+1)d} = \frac{\delta}{8}.
\end{equation}
Combining (\ref{e:phipm}) and (\ref{e:phipm2}) we obtain
\begin{equation}\label{e:phipm3}
 -\frac{\delta}{8} \leq \frac{X_{1}^{(1)}}{a_{n}} \leq M_{n}^{(1)} \Big( \frac{i^{*}}{n} \Big) = \bigvee_{s=1}^{i^{*}}\frac{X_{s}^{(1)}}{a_{n}}  \leq \frac{\delta}{8},
 \end{equation}
and thus
$$ \Big| W_{n}^{(1)} \Big( \frac{i^{*}}{n} \Big) - M_{n}^{(1)} \Big( \frac{i^{*}}{n} \Big) \Big| \leq  \Big| W_{n}^{(1)} \Big( \frac{i^{*}}{n} \Big) \Big| + \Big| M_{n}^{(1)} \Big( \frac{i^{*}}{n} \Big)\Big| \leq \frac{\delta}{8(m+1)d} + \frac{\delta}{8} \leq \frac{\delta}{4},$$
which is in contradiction with (\ref{e:i*}).

Therefore $W_{n}^{(1)}(i^{*}/n) > \delta / [8(m+1)d]$, and hence there exist $k \in \{1,\ldots,i^{*}\}$ and $j_{0} \in \{1,\ldots,d\}$ such that
\begin{equation}\label{e:est31}
 W_{n}^{(1)} \Big( \frac{i^{*}}{n} \Big) = \frac{1}{a_{n}} \Big(D^{1,j_{0}}_{+}Z_{k}^{(j_{0})+} \vee D^{1,j_{0}}_{-}Z_{k}^{(j_{0})-}\Big) > \frac{\delta}{8(m+1)d}.
\end{equation}
This implies
$$ \frac{D\|Z_{k}\|}{a_{n}} = \frac{D}{a_{n}} \bigvee_{j=1}^{d}|Z_{k}^{(j)}| \geq \frac{D}{a_{n}}|Z_{k}^{(j_{0})}| \geq  \frac{1}{a_{n}} \Big(D^{1,j_{0}}_{+}Z_{k}^{(j_{0})+} \vee D^{1,j_{0}}_{-}Z_{k}^{(j_{0})-}\Big) >  \frac{\delta}{8(m+1)d}.$$
From this, since $H_{n, 1}^{c} \cap H_{n, 2}^{c} \cap H_{n, 3}^{c}$ occurs, it follows that $m+1 \leq k \leq n-m$,
\begin{equation}\label{e:est51}
\frac{ D\|Z_{l}\|}{a_{n}}  \leq \frac{\delta}{8(m+1)d} \qquad \textrm{for all} \ l \in \{k-m,\ldots,k+m\} \setminus \{k\},
\end{equation}
and
\begin{equation}\label{e:est51-1}
\frac{ D|Z_{k}^{(p)}|}{a_{n}}  \leq \frac{\delta}{8(m+1)d} \qquad \textrm{for all} \ p \in \{1,\ldots,d\} \setminus \{j_{0}\}.
\end{equation}
The next step is to show that $M_{n}^{(1)}(i^{*}/n) = X_{j}^{(1)}/a_{n}$ for some $j \in \{1,\ldots,i^{*}\} \setminus \{k,\ldots,k+m\}$. If this is not the case, then $M_{n}^{(1)}(i^{*}/n) = X_{j}^{(1)}/a_{n}$ for some $j \in \{k,\ldots,k+m\}$ (with $j \leq i^{*}$). On the event $\{ Z_{k}^{(j_{0})}>0 \}$ it holds that
 $$  W_{n}^{(1)} \Big( \frac{i^{*}}{n} \Big) = \frac{1}{a_{n}} D^{1,j_{0}}_{+}Z_{k}^{(j_{0})+} = \frac{1}{a_{n}}C_{q;1,j_{0}}Z_{k}^{(j_{0})}$$
 for some $q \in \{0,\ldots, m\}$ (with $C_{q;1,j_{0}} \geq 0$), and now
we distinguish two cases:
\begin{itemize}
\item[(i)] $k+m \leq i^{*}$.
In this case $k+q \leq i^{*}$, and hence
\begin{equation}\label{e:est6}
\frac{X_{j}^{(1)}}{a_{n}} = M_{n}^{(1)}  \Big( \frac{i^{*}}{n} \Big) \geq \frac{X_{k+q}^{(1)}}{a_{n}}.
\end{equation}
Observe that we can write
 \begin{eqnarray*}
 \frac{X_{j}^{(1)}}{a_{n}} &=& \sum_{r=0}^{m} \sum_{p=1}^{d} \frac{C_{r;1,p} Z_{j-r}^{(p)}}{a_{n}} = \sum_{p=1}^{d} \frac{C_{j-k;1,p} Z_{k}^{(p)}}{a_{n}} + \sum_{\scriptsize \begin{array}{c}
                          r=0  \\[-0.1em]
                          r \neq j-k
                        \end{array}}^{m} \sum_{p=1}^{d} \frac{C_{r;1,p} Z_{j-r}^{(p)}}{a_{n}}\\[0.5em]
 & = & \frac{C_{j-k;1,j_{0}} Z_{k}^{(j_{0})}}{a_{n}}  + \sum_{\scriptsize \begin{array}{c}
                          p=1  \\[-0.1em]
                          p \neq j_{0}
                        \end{array}}^{d} \frac{C_{j-k;1,p} Z_{k}^{(p)}}{a_{n}} + \sum_{\scriptsize \begin{array}{c}
                          r=0  \\[-0.1em]
                          r \neq j-k
                        \end{array}}^{m} \sum_{p=1}^{d} \frac{C_{r;1,p} Z_{j-r}^{(p)}}{a_{n}}\\[0.5em]
 &=:& \frac{C_{j-k;1,j_{0}} Z_{k}^{(j_{0})}}{a_{n}} + F_{1},
 \end{eqnarray*}
and
 \begin{eqnarray*}
 \frac{X_{k+q}^{(1)}}{a_{n}} &=&  \frac{C_{q;1,j_{0}} Z_{k}^{(j_{0})}}{a_{n}}  + \sum_{\scriptsize \begin{array}{c}
                          p=1  \\[-0.1em]
                          p \neq j_{0}
                        \end{array}}^{d} \frac{C_{q;1,p} Z_{k}^{(p)}}{a_{n}} + \sum_{\scriptsize \begin{array}{c}
                          r=0  \\[-0.1em]
                          r \neq q
                        \end{array}}^{m} \sum_{p=1}^{d} \frac{C_{r;1,p} Z_{k+q-r}^{(p)}}{a_{n}}\\[0.5em]
 &=:& \frac{C_{q;1,j_{0}} Z_{k}^{(j_{0})}}{a_{n}} + F_{2}.
 \end{eqnarray*}
Relations (\ref{e:est51}) and (\ref{e:est51-1}) yield
\begin{eqnarray*}
|F_{1}| & \leq &  \sum_{\scriptsize \begin{array}{c}
                          p=1  \\[-0.1em]
                          p \neq j_{0}
                        \end{array}}^{d} \frac{D |Z_{k}^{(p)}|}{a_{n}} + \sum_{\scriptsize \begin{array}{c}
                          r=0  \\[-0.1em]
                          r \neq j-k
                        \end{array}}^{m} \sum_{p=1}^{d} \frac{D  \|Z_{j-r}\|}{a_{n}}\\[0.5em]
   & \leq & (d-1) \frac{\delta}{8(m+1)d} + md \frac{\delta}{8(m+1)d} < \frac{\delta}{4}.
\end{eqnarray*}
and similarly $|F_{2}| < \delta/4$.
Since by the definition of $D^{1,j_{0}}_{+}$ it holds that
$ C_{q;1,j_{0}} - C_{j-k;1,j_{0}} = D^{1,j_{0}}_{+} - C_{j-k;1,j_{0}} \geq 0$,
 using (\ref{e:est6}) we obtain
$$ 0 \leq  (C_{q;1,j_{0}} - C_{j-k;1,j_{0}})\frac{Z_{k}^{(j_{0})}}{a_{n}} =  \frac{X_{k+q}^{(1)}}{a_{n}} -F_{2} -  \frac{X_{j}^{(1)}}{a_{n}} + F_{1}  \leq |F_{1}| + |F_{2}| < \frac{\delta}{2}.$$
This implies
\begin{eqnarray*}
\Big| \frac{C_{q;1,j_{0}} Z_{k}^{(j_{0})}}{a_{n}} - \frac{X_{j}^{(1)}}{a_{n}} \Big| &=& \Big| \frac{C_{q;1,j_{0}} Z_{k}^{(j_{0})}}{a_{n}} -  \frac{C_{j-k;1,j_{0}}Z_{k}^{(j_{0})}}{a_{n}} - F_{1} \Big|\\[0.5em]
& \leq & (C_{q;1,j_{0}} - C_{j-k;1,j_{0}})\frac{Z_{k}^{(j_{0})}}{a_{n}} + |F_{1}|\\[0.4em]
 &<& \frac{\delta}{2} + \frac{\delta}{4} = \frac{3\delta}{4},
\end{eqnarray*}
but on the other hand, by (\ref{e:i*}) it holds that
$$ \Big| \frac{C_{q;1,j_{0}} Z_{k}^{(j_{0})}}{a_{n}} - \frac{X_{j}^{(1)}}{a_{n}} \Big| = \Big| W_{n}^{(1)} \Big( \frac{i^{*}}{n} \Big) - M_{n}^{(1)} \Big( \frac{i^{*}}{n} \Big) \Big| > \delta.$$
Hence we arrive at a contradiction, and this case can not happen.
\item[(ii)] $k+m > i^{*}$. Note that in this case $k \leq j \leq i^{*} < k+m$. Since
 $$ M_{n}^{(1)} \Big( \frac{k+m}{n} \Big) = \bigvee_{s=1}^{k+m}\frac{X_{s}^{(1)}}{a_{n}} \geq M_{n}^{(1)}  \Big( \frac{i^{*}}{n} \Big) = \frac{X_{j}^{(1)}}{a_{n}},$$
 it holds that
 $$ M_{n}^{(1)} \Big( \frac{k+m}{n} \Big) = \frac{X_{s_{0}}^{(1)}}{a_{n}},$$
 for some $s_{0} \in \{j, \ldots, k+m \} \subseteq \{k, \ldots, k+m\}$.
 Then
\begin{eqnarray*}
 \frac{X_{s_{0}}^{(1)}}{a_{n}} &=& \frac{C_{s_{0}-k;1,j_{0}} Z_{k}^{(j_{0})}}{a_{n}}  + \sum_{\scriptsize \begin{array}{c}
                          p=1  \\[-0.1em]
                          p \neq j_{0}
                        \end{array}}^{d} \frac{C_{s_{0}-k;1,p} Z_{k}^{(p)}}{a_{n}} + \sum_{\scriptsize \begin{array}{c}
                          r=0  \\[-0.1em]
                          r \neq s_{0}-k
                        \end{array}}^{m} \sum_{p=1}^{d} \frac{C_{r;1,p} Z_{s_{0}-r}^{(p)}}{a_{n}}\\[0.4em]
 &=:& \frac{C_{s_{0}-k;1,j_{0}} Z_{k}^{(j_{0})}}{a_{n}} + F_{3},
\end{eqnarray*}
with $|F_{3}| < \delta/4$, which holds by relations (\ref{e:est51}) and (\ref{e:est51-1}). By relation (\ref{e:i*q}) we have
\begin{equation}\label{e:auxineq1}
 \Big| \frac{C_{q;1,j_{0}} Z_{k}^{(j_{0})}}{a_{n}} - \frac{X_{s_{0}}^{(1)}}{a_{n}} \Big| = \Big| W_{n}^{(1)} \Big( \frac{i^{*}}{n} \Big) - M_{n}^{(1)} \Big( \frac{k+m}{n} \Big) \Big| > \delta,
\end{equation}
and repeating the arguments as in (i), but with
$$ \frac{X_{s_{0}}^{(1)}}{a_{n}} = M_{n}^{(1)}  \Big( \frac{k+m}{n} \Big) \geq \frac{X_{k+q}^{(1)}}{a_{n}}$$
instead of (\ref{e:est6}), we arrive at
$$ \delta < \Big| \frac{C_{q;1,j_{0}} Z_{k}^{(j_{0})}}{a_{n}} -  \frac{C_{s_{0}-k;1,j_{0}}Z_{k}^{(j_{0})}}{a_{n}} - F_{3} \Big| \leq (C_{q;1,j_{0}} - C_{s_{0}-k;1,j_{0}})\frac{Z_{k}^{(j_{0})}}{a_{n}} + |F_{3}| <  \frac{3\delta}{4},$$
which is not possible, and we conclude that this case also can not happen.
\end{itemize}
One can similarly handle the event $\{ Z_{k}^{(j_{0})} <0 \}$ to arrive at a contradiction. Therefore indeed $M_{n}^{(1)}(i^{*}/n) = X_{j}^{(1)}/a_{n}$ for some $j \in \{1,\ldots,i^{*}\} \setminus \{k,\ldots,k+m\}$.

Now we have four cases: (A1) all random vectors $Z_{j-m}, \ldots, Z_{j}$ are "small", (A2) exactly one is "large" with exactly one "large" component, (A3) exactly one is "large" with at least two "large" components and (A4) at least two of them are "large", where we say $Z$ is "large" if $ D\|Z\| /a_{n} > \delta/[8(m+1)d]$, otherwise it is "small", and similarly the component $Z^{(s)}$ is "large" if $D |Z^{(s)}| /a_{n} > \delta/[8(m+1)d] $. We will show that the first two cases are not possible.
\begin{itemize}
  \item[(A1)] $ D\|Z_{j-l}\|/a_{n} \leq \delta/[8(m+1)d]$ for every $l=0,\ldots,m$.
 This yields (as in (\ref{e:phipm2}))
$$ \Big| M_{n}^{(1)} \Big( \frac{i^{*}}{n} \Big) \Big| = \frac{|X_{j}^{(1)}|}{a_{n}} \leq  \frac{\delta}{8}.$$
Let $q$ be as above (on the event $\{ Z_{k}^{(j_{0})}>0 \}$), i.e.
  $$  W_{n}^{(1)} \Big( \frac{i^{*}}{n} \Big) = \frac{1}{a_{n}} D^{1,j_{0}}_{+}Z_{k}^{(j_{0})+} = \frac{1}{a_{n}}C_{q;1,j_{0}}Z_{k}^{(j_{0})}.$$
 If $k+m \leq i^{*}$, then according to (\ref{e:est6}),
\begin{equation}\label{e:A1case}
 \frac{X_{j}^{(1)}}{a_{n}} \geq \frac{X_{k+q}^{(1)}}{a_{n}} = \frac{C_{q;1,j_{0}} Z_{k}^{(j_{0})}}{a_{n}} + F_{2},
\end{equation}
where $F_{2}$ is as in (i) above, with $|F_{2}| < \delta/4$.
This implies
$$  \frac{C_{q;1,j_{0}}Z_{k}^{(j_{0})}}{a_{n}} \leq \frac{X_{j}^{(1)}}{a_{n}} - F_{2} \leq \frac{|X_{j}^{(1)}|}{a_{n}} + |F_{2}| < \frac{\delta}{8} + \frac{\delta}{4} = \frac{3\delta}{8},$$
and hence
\begin{eqnarray*}
 \Big| W_{n}^{(1)} \Big( \frac{i^{*}}{n} \Big) - M_{n}^{(1)} \Big( \frac{i^{*}}{n} \Big) \Big| &=& \Big| \frac{C_{q;1,j_{0}} Z_{k}^{(j_{0})}}{a_{n}} - \frac{X_{j}^{(1)}}{a_{n}} \Big|\\[0.4em]
 &\leq &  \frac{C_{q;1,j_{0}} Z_{k}^{(j_{0})}}{a_{n}} + \frac{|X_{j}^{(1)}|}{a_{n}}
  < \frac{3\delta}{8} + \frac{\delta}{8} = \frac{\delta}{2},
\end{eqnarray*}
 which is in contradiction with (\ref{e:i*}). On the other hand, if $k+m > i^{*}$, we have two possibilities: $M_{n}^{(1)}((k+m)/n) = M_{n}^{(1)}(i^{*}/n)$ or $M_{n}^{(1)}((k+m)/n) > M_{n}^{(1)}(i^{*}/n)$. When $M_{n}^{(1)}((k+m)/n) = M_{n}^{(1)}(i^{*}/n) = X_{j}^{(1)}/a_{n}$, since $k+q \leq k+m$ note that relation (\ref{e:A1case}) holds, and similarly as above we obtain
 $$ \Big| W_{n}^{(1)} \Big( \frac{i^{*}}{n} \Big) - M_{n}^{(1)} \Big( \frac{k+m}{n} \Big) \Big| = \Big| \frac{C_{q;1,j_{0}} Z_{k}^{(j_{0})}}{a_{n}} - \frac{X_{j}^{(1)}}{a_{n}} \Big| <  \frac{\delta}{2},$$
  which is in contradiction with (\ref{e:i*q}). Alternatively, when $M_{n}^{(1)}((k+m)/n) > M_{n}^{(1)}(i^{*}/n)$, then it holds that  $M_{n}((k+m)/n) = X_{s_{0}}^{(1)}/a_{n}$ for some $s_{0} \in \{i^{*}, \ldots, k+m\}$. Now in the same manner as in (ii) above we get a contradiction. We handle the event $\{ Z_{k}^{(j_{0})} <0 \}$ similarly to arrive at a contradiction, and therefore this case can not happen.

  \item[(A2)]  There exist $l_{1} \in \{0,\ldots,m\}$ and $s_{1} \in \{1,\ldots,d\}$ such that $ D|Z_{j-l_{1}}^{(s_{1})}|/a_{n} > \delta/[8(m+1)d]$, $D|Z_{j-l_{1}}^{(s)}|/a_{n} \leq \delta/[8(m+1)d]$ for every $s \in \{1,\ldots,d\} \setminus \{s_{0}\}$, and $ D\|Z_{j-l}\|/a_{n} \leq \delta/[8(m+1)d]$ for every $l \in \{0,\ldots,m\} \setminus \{l_{1}\}$. Here we analyze only what happens on the event $\{Z_{k}^{(j_{0})}>0\}$ (the event $\{Z_{k}^{(j_{0})}<0\}$ can be treated analogously and is therefore omitted). Assume first $k+m \leq i^{*}$. Then
\begin{equation}\label{e:est7}
\frac{X_{j}^{(1)}}{a_{n}} \geq \frac{X_{k+q}^{(1)}}{a_{n}} = \frac{C_{q;1,j_{0}} Z_{k}^{(j_{0})}}{a_{n}} + F_{2},
\end{equation}
where $q$ and $F_{2}$ are as in (i) above, with $|F_{2}| < \delta/4$.
 Write
\begin{eqnarray*}
\frac{X_{j}^{(1)}}{a_{n}} &=& \frac{C_{l_{1};1,s_{1}} Z_{j-l_{1}}^{(s_{1})}}{a_{n}}  + \sum_{\scriptsize \begin{array}{c}
                          s=1  \\[-0.1em]
                          s \neq s_{1}
                        \end{array}}^{d} \frac{C_{l_{1};1,s} Z_{j-l_{1}}^{(s)}}{a_{n}} + \sum_{\scriptsize \begin{array}{c}
                          r=0  \\[-0.1em]
                          r \neq l_{1}
                        \end{array}}^{m} \sum_{s=1}^{d} \frac{C_{r;1,p} Z_{j-r}^{(s)}}{a_{n}}\\[0.2em]
 &=:& \frac{C_{l_{1};1,s_{1}} Z_{j-l_{1}}^{(s_{1})}}{a_{n}} + F_{4}.
\end{eqnarray*}
Similarly as before for $F_{1}$ we obtain $|F_{4}| < \delta/4$. Since $j-l_{1} \leq j \leq i^{*}$, by the definition of the process $W_{n}$ we have
$$ \frac{C_{q;1,j_{0}}Z_{k}^{(j_{0})}}{a_{n}} =  W_{n}^{(1)} \Big( \frac{i^{*}}{n} \Big) \geq \frac{1}{a_{n}} \Big(D^{1,s_{1}}_{+}Z_{j-l_{1}}^{(s_{1})+} \vee D^{1,s_{1}}_{-}Z_{j-l_{1}}^{(s_{1})-}\Big)
 \geq \frac{C_{l_{1};1,s_{1}}Z_{j-l_{1}}^{(s_{1})}}{a_{n}},$$
and this yields
\begin{equation}\label{e:est8}
\frac{ C_{q;1,j_{0}} Z_{k}^{(j_{0})}}{a_{n}} - \frac{X_{j}^{(1)}}{a_{n}} \geq \frac{C_{l_{1};1, s_{1}}Z_{j-l_{1}}^{(s_{1})}}{a_{n}} - \frac{X_{j}^{(1)}}{a_{n}} = - F_{4}.
\end{equation}
Relations (\ref{e:est7}) and (\ref{e:est8}) yield
$$- (|F_{2}| + |F_{4}|) \leq -F_{4} \leq \frac{C_{q;1,j_{0}} Z_{k}^{(j_{0})}}{a_{n}} - \frac{X_{j}^{(1)}}{a_{n}} \leq -F_{2} \leq  |F_{2}| + |F_{4}|,$$
and this implies
$$ \Big| W_{n}^{(1)} \Big( \frac{i^{*}}{n} \Big) - M_{n}^{(1)} \Big( \frac{i^{*}}{n} \Big) \Big| = \Big| \frac{C_{q;1,j_{0}} Z_{k}^{(j_{0})}}{a_{n}} - \frac{X_{j}^{(1)}}{a_{n}} \Big| \leq |F_{2}| + |F_{4}| < \frac{\delta}{4} + \frac{\delta}{4}=\frac{\delta}{2},$$
which is in contradiction with (\ref{e:i*}).
 Assume now $k+m > i^{*}$. If $M_{n}^{(1)}((k+m)/n) = M_{n}^{(1)}(i^{*}/n) = X_{j}^{(1)}/a_{n}$, relation (\ref{e:est7}) still holds and this leads to
 $$ \Big| W_{n}^{(1)} \Big( \frac{i^{*}}{n} \Big) - M_{n}^{(1)} \Big( \frac{k+m}{n} \Big) \Big| = \Big| \frac{C_{q;1,j_{0}} Z_{k}^{(j_{0})}}{a_{n}} - \frac{X_{j}^{(1)}}{a_{n}} \Big| <  \frac{\delta}{2},$$
  which is in contradiction with (\ref{e:i*q}). On the other hand, if $M_{n}^{(1)}((k+m)/n) > M_{n}^{(1)}(i^{*}/n)$, then $M_{n}^{(1)}((k+m)/n) = X_{s_{0}}^{(1)}/a_{n}$ for some $s_{0} \in \{i^{*}, \ldots, k+m\}$, and with the same arguments as in (ii) above
 we obtain
a contradiction with relation (\ref{e:i*q}).
Hence this case also can not happen.

  \item[(A3)] There exist $l_{1} \in \{0,\ldots,m\}$ and $s_{1}, s_{2} \in \{1,\ldots,d\}$ such that $s_{1} \neq s_{2}$,
  $D|Z_{j-l_{1}}^{(s_{1})}|/a_{n} > \delta/[8(m+1)d]$, $ D|Z_{j-l_{1}}^{(s_{2})}|/a_{n} > \delta/[8(m+1)d]$ and $D\|Z_{j-l}\|/a_{n} \leq \delta/[8(m+1)d]$ for all $l \in \{0,\ldots,m\} \setminus \{l_{1}\}$. Note that in this case the event $H_{n, 4}$ occurs (with $l_{2}=l_{1}$).

  \item[(A4)] There exist $l_{1}, l_{2} \in \{0,\ldots,m\}$ such that $l_{1} \neq l_{2}$,
  $D\|Z_{j-l_{1}}\|/a_{n} > \delta/[8(m+1)d]$ and $D\|Z_{j-l_{2}}\|/a_{n} > \delta/[8(m+1)d]$. The definition of the max-norm $\|\cdot\|$ implies the existence of $s_{1}, s_{2} \in \{1,\ldots,d\}$ such that $D|Z_{j-l_{1}}^{(s_{1})}|/a_{n} > \delta/[8(m+1)d]$ and $D|Z_{j-l_{2}}^{(s_{2})}|/a_{n} > \delta/[8(m+1)d]$. Hence the event $H_{n, 4}$ occurs also in this case.
\end{itemize}

Therefore only the cases (A3) and (A4) are possible, and this yields $\widehat{D}_{n} \subseteq H_{n, 4}$. Hence (\ref{e:estim1}) holds.
By stationarity we have
\begin{equation}\label{e:Hn1new}
 \Pr(H_{n, 1}) \leq (3m+1) \Pr \bigg( \frac{D\|Z_{1}\|}{a_{n}} > \frac{\delta}{8(m+1)d} \bigg).
\end{equation}
For an arbitrary $c \in \mathbb{R}$, $c>0$, it holds that
   \begin{eqnarray*}
     \Pr \bigg( \frac{D\|Z_{1}\|}{a_{n}} > \frac{\delta}{8(m+1)d} \bigg) & &\\[0.6em]
      & \hspace*{-16em} =& \hspace*{-8em} \Pr \bigg( \frac{D\|Z_{1}\|}{a_{n}} > \frac{\delta}{8(m+1)d},\,D > c \bigg) + \Pr \bigg( \frac{D\|Z_{1}\|}{a_{n}} > \frac{\delta}{8(m+1)d},\,D \leq c \bigg)\\[0.6em]
       & \hspace*{-16em} \leq & \hspace*{-8em} \Pr(D > c) + \Pr \bigg( \frac{\|Z_{1}\|}{a_{n}} > \frac{\delta}{8(m+1)dc} \bigg).
   \end{eqnarray*}
By the regular variation property (i.e.~relation (\ref{e:nizrv}))
\begin{equation*}
\lim_{n \to \infty}  \Pr \bigg( \frac{\|Z_{1}\|}{a_{n}} > \frac{\delta}{8(m+1)dc} \bigg) =0,
\end{equation*}
and therefore from (\ref{e:Hn1new}) we get
$ \limsup_{n \to \infty} \Pr (H_{n,1}) \leq (3m+1) \Pr (D > c).$
Letting $c \to \infty$ we conclude
\begin{equation}\label{e:est2}
\lim_{n \to \infty} \Pr(H_{n, 1})=0.
\end{equation}
As for $H_{n, 2}$, by stationarity we have
\begin{eqnarray}
 \nonumber \Pr(H_{n, 2} \cap \{D \leq c\}) &&\\[0.5em]
 \nonumber & \hspace*{-12em} =& \hspace*{-6em} \sum_{k=1}^{n} \sum_{\scriptsize \begin{array}{c}
                          l=k-m  \\[-0.1em]
                          l \neq k
                        \end{array}}^{k+m} \Pr \bigg( \frac{D\|Z_{k}\|}{a_{n}} > \frac{\delta}{8(m+1)d},\,\frac{D\|Z_{l}\|}{a_{n}} > \frac{\delta}{8(m+1)d},\,D \leq c \bigg)\\[-0.1em]
   \nonumber & \hspace*{-12em} \leq & \hspace*{-6em} \sum_{k=1}^{n} \sum_{\scriptsize \begin{array}{c}
                          l=k-m  \\[-0.1em]
                          l \neq k
                        \end{array}}^{k+m} \Pr \bigg( \frac{\|Z_{k}\|}{a_{n}} > \frac{\delta}{8(m+1)dc},\,\frac{\|Z_{l}\|}{a_{n}} > \frac{\delta}{8(m+1)dc} \bigg)\\[0.1em]
   \nonumber & \hspace*{-12em} \leq & \hspace*{-6em} 2n \sum_{i=1}^{m} \Pr \bigg( \frac{\|Z_{0}\|}{a_{n}} > \frac{\delta}{8(m+1)dc},\,\frac{\|Z_{i}\|}{a_{n}} > \frac{\delta}{8(m+1)dc} \bigg)\\[0.1em]
   \nonumber & \hspace*{-12em} \leq & \hspace*{-6em} 2 \sum_{i=1}^{m} n \Pr \bigg( \frac{\|Z_{0}\|}{a_{n}} > \frac{\delta}{8(m+1)dc} \bigg) \frac{\Pr \Big( \frac{\|Z_{0}\|}{a_{n}} > \frac{\delta}{8(m+1)dc},\,\frac{\|Z_{i}\|}{a_{n}} > \frac{\delta}{8(m+1)dc} \Big)}{\Pr \Big( \frac{\|Z_{0}\|}{a_{n}} > \frac{\delta}{8(m+1)dc} \Big)}
 \end{eqnarray}
 for an arbitrary $c>0$. Therefore relation (\ref{e:nizrv}) and the asymptotical independence condition (\ref{e:asyind}) yield
$ \lim_{n \to \infty} \Pr(H_{n, 2} \cap \{ D \leq c \})=0$, and this implies
 $$ \limsup_{n \to \infty}  \Pr(H_{n, 2}) \leq \limsup_{n \to \infty} \Pr(H_{n, 2} \cap \{D > c\}) \leq \Pr (D > c).$$
Letting again $c \to \infty$ we conclude
\begin{equation}\label{e:est5}
\lim_{n \to \infty} \Pr(H_{n, 2})=0.
\end{equation}
By the definition of the set $H_{n, 3}$ and stationarity it holds that
\begin{eqnarray}
 \nonumber \Pr(H_{n, 3} \cap \{D \leq c\}) &&\\[0.5em]
 \nonumber & \hspace*{-14em} =& \hspace*{-7em} \sum_{k=1}^{n} \sum_{\scriptsize \begin{array}{c}
                          l,s=1  \\[-0.1em]
                          l \neq s
                        \end{array}}^{d} \Pr \bigg( \frac{D|Z_{k}^{(l)}|}{a_{n}} > \frac{\delta}{8(m+1)d},\,\frac{D|Z_{k}^{(s)}|}{a_{n}} > \frac{\delta}{8(m+1)d},\,D \leq c \bigg)\\[-0.1em]
   \nonumber & \hspace*{-14em} \leq & \hspace*{-7em}  \sum_{\scriptsize \begin{array}{c}
                          l,s=1  \\[-0.1em]
                          l \neq s
                        \end{array}}^{d} n \Pr \bigg( \frac{|Z_{1}^{(l)}|}{a_{n}} > \frac{\delta}{8(m+1)dc},\,\frac{|Z_{1}^{(s)}|}{a_{n}} > \frac{\delta}{8(m+1)dc} \bigg)\\[-0.1em]
   \nonumber & \hspace*{-14em} = & \hspace*{-7em}  \sum_{\scriptsize \begin{array}{c}
                          l,s=1  \\[-0.1em]
                          l \neq s
                        \end{array}}^{d} n \Pr \bigg( \frac{|Z_{1}^{(s)}|}{a_{n}} > \frac{\delta}{8(m+1)dc} \bigg) \Pr \bigg( \frac{|Z_{1}^{(l)}|}{a_{n}} > \frac{\delta}{8(m+1)dc}\,\bigg|\,\frac{|Z_{1}^{(s)}|}{a_{n}} > \frac{\delta}{8(m+1)dc} \bigg)\\[-0.1em]
  \nonumber & \hspace*{-14em} \leq & \hspace*{-7em}  \sum_{\scriptsize \begin{array}{c}
                          l,s=1  \\[-0.1em]
                          l \neq s
                        \end{array}}^{d} n \Pr \bigg( \frac{\|Z_{1}\|}{a_{n}} > \frac{\delta}{8(m+1)dc} \bigg) \Pr \bigg( \frac{|Z_{1}^{(l)}|}{a_{n}} > \frac{\delta}{8(m+1)dc}\,\bigg|\,\frac{|Z_{1}^{(s)}|}{a_{n}} > \frac{\delta}{8(m+1)dc} \bigg).
 \end{eqnarray}
  By relations (\ref{e:asyindcomp}) and (\ref{e:nizrv}) it follows that $ \lim_{n \to \infty} \Pr(H_{n, 3} \cap \{ D \leq c \})=0$, and then letting $c \to \infty$ we obtain
\begin{equation}\label{e:est10}
\lim_{n \to \infty} \Pr(H_{n, 3})=0.
\end{equation}
 Note that
\begin{eqnarray}
  \nonumber H_{n, 4} & \subseteq &  \bigg\{ \exists\,j \in \{1,\ldots,n\}, \ \exists\,l_{1}, l_{2} \in \{0,\ldots,m\}
   \ \textrm{and} \ \exists\,s_{1}, s_{2} \in \{1,\ldots,d\} \ \textrm{such that}\\[0.1em]
   \nonumber && \ (l_{1},s_{1}) \neq (l_{2}, s_{2}), \ \frac{ D|Z_{j-l_{1}}^{(s_{1})}|}{a_{n}} > \frac{\delta}{8(m+1)d} \ \textrm{and} \  \frac{ D|Z_{j-l_{2}}^{(s_{2})}|}{a_{n}} > \frac{\delta}{8(m+1)d}  \bigg\},
\end{eqnarray}
which implies
 \begin{eqnarray}
 \nonumber \Pr(H_{n, 4} \cap \{D \leq c\}) &&\\[0.5em]
 \nonumber & \hspace*{-14em} =& \hspace*{-7em} \sum_{j=1}^{n} \sum_{\scriptsize \begin{array}{c}
                          l_{1},l_{2}=0  \\[-0.1em]
                          l_{1} \neq l_{2}
                        \end{array}}^{m}
                        \sum_{s_{1},s_{2}=1}^{d} \Pr \bigg( \frac{D|Z_{j-l_{1}}^{(s_{1})}|}{a_{n}} > \frac{\delta}{8(m+1)d},\,\frac{D|Z_{j-l_{2}}^{(s_{2})}|}{a_{n}} > \frac{\delta}{8(m+1)d},\,D \leq c \bigg)\\[-0.1em]
  \nonumber & \hspace*{-12em} + & \hspace*{-6em} \sum_{j=1}^{n} \sum_{l=0}^{m}
                        \sum_{\scriptsize \begin{array}{c}
                          s_{1},s_{2}=1\\[-0.1em]
                          s_{1} \neq s_{2}
                        \end{array}}^{d} \Pr \bigg( \frac{D|Z_{j-l}^{(s_{1})}|}{a_{n}} > \frac{\delta}{8(m+1)d},\,\frac{D|Z_{j-l}^{(s_{2})}|}{a_{n}} > \frac{\delta}{8(m+1)d},\,D \leq c \bigg)\\[-0.1em]
  \nonumber & \hspace*{-14em} \leq & \hspace*{-7em} \sum_{j=1}^{n} \sum_{\scriptsize \begin{array}{c}
                          l_{1},l_{2}=0  \\[-0.1em]
                          l_{1} \neq l_{2}
                        \end{array}}^{m}
                        d^{2} \Pr \bigg( \frac{\|Z_{j-l_{1}}\|}{a_{n}} > \frac{\delta}{8(m+1)dc},\,\frac{\|Z_{j-l_{2}}\|}{a_{n}} > \frac{\delta}{8(m+1)dc} \bigg)\\[-0.1em]
   \nonumber & \hspace*{-12em} + & \hspace*{-6em} \sum_{j=1}^{n} \sum_{l=0}^{m}
                        \sum_{\scriptsize \begin{array}{c}
                          s_{1},s_{2}=1\\[-0.1em]
                          s_{1} \neq s_{2}
                        \end{array}}^{d} \Pr \bigg( \frac{|Z_{j-l}^{(s_{1})}|}{a_{n}} > \frac{\delta}{8(m+1)dc},\,\frac{|Z_{j-l}^{(s_{2})}|}{a_{n}} > \frac{\delta}{8(m+1)dc} \bigg)
 \end{eqnarray}
 As before for $H_{n,2}$ and $H_{n,3}$ by relations (\ref{e:asyind}), (\ref{e:asyindcomp}) and (\ref{e:nizrv}) it follows that $ \lim_{n \to \infty} \Pr(H_{n, 4} \cap \{ D \leq c \})=0$, and again letting $c \to \infty$ yields
\begin{equation}\label{e:est10-1}
\lim_{n \to \infty} \Pr(H_{n, 4})=0.
\end{equation}
Now from (\ref{e:estim1}) and (\ref{e:est2})--(\ref{e:est10-1}) we obtain
$ \lim_{n \to \infty} \Pr(D_{n})=0,$
and hence (\ref{e:Yn}) yields
\begin{equation}\label{eq:Ynend}
\lim_{n \to \infty} \Pr(R_{n}> \delta)=0.
\end{equation}

It remains to estimate the second term on the right-hand side of (\ref{eq:AB}). Let
$$ E_{n} = \{\exists\,v \in \Gamma_{M_{n}^{(1)}} \ \textrm{such that} \ d(v,z) > \delta \ \textrm{for every} \ z \in \Gamma_{ W_{n}^{(1)}} \}.$$
Then by the definition of $T_{n}$
\begin{equation}\label{e:Tnfirst}
 \{T_{n} > \delta\} \subseteq E_{n}.
\end{equation}
On the event $E_{n}$ it holds that $d(v, \Gamma_{ W_{n}^{(1)}})> \delta$.
Interchanging the roles of the processes $M_{n}^{(1)}$ and $W_{n}^{(1)}$, in the same way as before for the event $D_{n}$ it can be shown that
\begin{equation}\label{e:i*qneg}
 \Big| W_{n}^{(1)} \Big( \frac{i^{*}-l}{n} \Big) - M_{n}^{(1)} \Big( \frac{i^{*}}{n} \Big) \Big| > \delta
\end{equation}
for all $l=0,\ldots, m$ (such that $i^{*}-l \geq 0$), where $i^{*}=\floor{nt_{v}}$ or $i^{*}=\floor{nt_{v}}-1$, and $v=(t_{v},x_{v})$.

Similarly as before we want to show that $E_{n} \cap (H_{n, 1} \cup H_{n, 2} \cup H_{n, 3})^{c} \subseteq H_{n, 4}$. Assume the event $E_{n} \cap (H_{n, 1} \cup H_{n, 2} \cup H_{n, 3})^{c}$ occurs. Note that (\ref{e:i*qneg}) (for $l=0$) is in fact (\ref{e:i*}), and hence by repeating the arguments used for $D_{n}$ we conclude that (\ref{e:est31}) holds, i.e.
$$W_{n}^{(1)} \Big( \frac{i^{*}}{n} \Big) = \frac{1}{a_{n}} \Big(D^{1,j_{0}}_{+}Z_{k}^{(j_{0})+} \vee D^{1,j_{0}}_{-}Z_{k}^{(j_{0})-}\Big) > \frac{\delta}{8(m+1)d}$$
for some $k \in \{1,\ldots, i^{*}\}$ and $j_{0} \in \{1,\ldots,d\}$.

Here we also claim that $M_{n}^{(1)}(i^{*}/n) = X_{j}^{(1)}/a_{n}$ for some $j \in \{1,\ldots,i^{*}\} \setminus \{k,\ldots,k+m\}$. Hence assume this is not the case, i.e. $M_{n}^{(1)}(i^{*}/n) = X_{j}^{(1)}/a_{n}$ for some $j \in \{k,\ldots,k+m\}$ (with $j \leq i^{*}$). We can repeat the arguments from (i) above to conclude that $k + m \leq i^{*}$ is not possible. It remains to see what happens when $k + m > i^{*}$. For this case we can not simply repeat the procedure from (ii), since now we do not have the inequality
$$\Big| W_{n}^{(1)} \Big( \frac{i^{*}}{n} \Big) - M_{n}^{(1)} \Big( \frac{k+m}{n} \Big) \Big| > \delta$$
from (\ref{e:auxineq1}) that we used in (ii). Let
 $$W_{n}^{(1)} \Big( \frac{i^{*}-m}{n} \Big) = \frac{1}{a_{n}} \Big(D^{1,j_{0}^{*}}_{+}Z_{k^{*}}^{(j_{0}^{*})+} \vee D^{1,j_{0}^{*}}_{-}Z_{k^{*}}^{(j_{0}^{*})-}\Big) = \frac{1}{a_{n}} C_{q^{*};1,j_{0}^{*}}Z_{k^{*}}^{(j_{0}^{*})} $$
for some $k^{*} \in \{1, \ldots, i^{*}-m\}$, $j_{0}^{*} \in \{1,\ldots,d\}$ and $q^{*} \in \{0,\ldots, m\}$. Note that $i^{*}-m \geq 1$ since $m+1 \leq k  \leq i^{*}$ (the first inequality holds since $H_{n, 1}^{c}$ occurs). We distinguish two cases:
\begin{itemize}
\item[(a)] $W_{n}^{(1)}(i^{*}/n) > M_{n}^{(1)}(i^{*}/n)$. In this case the definition of $i^{*}$ implies that $M_{n}^{(1)}(i^{*}/n) \leq x_{v} \leq W_{n}^{(1)}(i^{*}/n)$. To see this observe that $v$ lies on the part of the completed graph $\Gamma_{M_{n}^{(1)}}$ above the time segment $[i^{*}/n, (i^{*}+1)/n]$. Since the process $M_{n}^{(1)}$ is non-decreasing, this implies $M_{n}^{(1)}(i^{*}/n) \leq x_{v}$. The second inequality follows from the fact that $x_{v} > W_{n}^{(1)}(i^{*}/n)$ would imply $t_{v}=(i^{*}+1)/n$ and $(t_{v}, W_{n}^{(1)}(i^{*}/n)) \in \Gamma_{W_{n}^{(1)}} \cap \Gamma_{M_{n}^{(1)}}$, i.e. $d(v, \Gamma_{ W_{n}^{(1)}})=0$, which is not possible on the event $E_{n}$.
     Now, since $|t_{v}-(i^{*}-m)/n| < (m+1)/n \leq \delta$,  from $d(v, \Gamma_{W_{n}^{(1)}})> \delta$ we obtain
$$ \widetilde{d} \Big((x_{v}, \Big[ W_{n}^{(1)} \Big( \frac{i^{*}-m}{n} \Big), W_{n}^{(1)} \Big(\frac{i^{*}}{n} \Big) \Big] \Big) > \delta,$$
where $\widetilde{d}$ is the Euclidean metric on $\mathbb{R}$. This yields
$W_{n}^{(1)} ((i^{*}-m)/n) > M_{n}^{(1)} (i^{*}/n)$,
and by (\ref{e:i*qneg}) we have
\begin{equation}\label{e:estE1}
 W_{n}^{(1)} \Big( \frac{i^{*}-m}{n} \Big) > M_{n}^{(1)} \Big( \frac{i^{*}}{n} \Big) + \delta.
\end{equation}
From this, taking into account relation (\ref{e:phipm3}), we obtain
$$ \frac{D\|Z_{k^{*}}\|}{a_{n}} \geq \frac{D |Z_{k^{*}}^{(j_{0}^{*})}|}{a_{n}} \geq  W_{n}^{(1)} \Big( \frac{i^{*}-m}{n} \Big) >  -\frac{\delta}{8} + \delta = \frac{7\delta}{8} > \frac{\delta}{8(m+1)d},$$
and since the event $H_{n, 2}^{c} \cap H_{n, 3}^{c}$ occurs it follows that
\begin{equation}\label{e:estE12}
\frac{D\|Z_{l}\|}{a_{n}} \leq \frac{\delta}{8(m+1)d} \quad \textrm{for every} \  l \in \{k^{*}-m,\ldots,k^{*}+m\} \setminus \{k^{*}\},
\end{equation}
and
\begin{equation}\label{e:estE12-1}
\frac{D|Z_{k^{*}}^{(p)}|}{a_{n}} \leq \frac{\delta}{8(m+1)d} \quad \textrm{for every} \  p \in \{1,\ldots,d \} \setminus \{j_{0}^{*}\}.
\end{equation}
Since $k^{*} + q^{*} \leq i^{*} - m + q^{*} \leq i^{*}$, it holds that
\begin{equation}\label{e:estE2}
 \frac{X_{j}^{(1)}}{a_{n}} = M_{n}^{(1)} \Big( \frac{i^{*}}{n} \Big) \geq \frac{X_{k^{*}+q^{*}}^{(1)}}{a_{n}} = \frac{C_{q^{*};1,j_{0}^{*}} Z_{k_{*}}^{(j_{0}^{*})}}{a_{n}} + F_{5},
\end{equation}
where
$$ F_{5} =  \sum_{\scriptsize \begin{array}{c}
                          p=1  \\[-0.1em]
                          p \neq j_{0}^{*}
                        \end{array}}^{d} \frac{C_{q^{*};1,p} Z_{k^{*}}^{(p)}}{a_{n}} + \sum_{\scriptsize \begin{array}{c}
                          r=0  \\[-0.1em]
                          r \neq q^{*}
                        \end{array}}^{m} \sum_{p=1}^{d} \frac{C_{r;1,p} Z_{k^{*}+q^{*}-r}^{(p)}}{a_{n}}.$$
By (\ref{e:estE1}) and (\ref{e:estE2}) we have
$$ \frac{1}{a_{n}} C_{q^{*};1,j_{0}^{*}}Z_{k^{*}}^{(j_{0}^{*})} > \frac{X_{j}^{(1)}}{a_{n}} + \delta \geq \frac{1}{a_{n}} C_{q^{*};1,j_{0}^{*}}Z_{k^{*}}^{(j_{0}^{*})} + F_{5} + \delta,$$
i.e.
$F_{5} < -\delta$. But this is not possible since relations (\ref{e:estE12}) and (\ref{e:estE12-1}) imply $|F_{5}| \leq \delta/4$,
and hence we conclude that this case can not happen.
\item[(b)] $W^{(1)}_{n}(i^{*}/n) \leq M_{n}^{(1)}(i^{*}/n)$. Then from (\ref{e:i*qneg}) we get
\begin{equation}\label{e:estE3}
 M_{n}^{(1)} \Big( \frac{k+m}{n} \Big) \geq M_{n}^{(1)} \Big( \frac{i^{*}}{n} \Big) \geq W_{n}^{(1)} \Big( \frac{i^{*}}{n} \Big) + \delta.
\end{equation}
Therefore
$$ \Big| W_{n}^{(1)} \Big( \frac{i^{*}}{n} \Big) - M_{n}^{(1)} \Big(\frac{k+m}{n} \Big) \Big| > \delta,$$
and repeating the arguments from (ii) above we conclude that this case also can not happen.
\end{itemize}
Therefore we have proved that  $M_{n}^{(1)}(i^{*}/n) = X_{j}^{(1)}/a_{n}$ for some $j \in \{1,\ldots,i^{*}\} \setminus \{k,\ldots,k+m\}$. Similar as before one can show that Cases (A1) and (A2) can not happen, and hence only Cases (A3) and (A4) are possible, which means that the event $H_{n, 4}$ occurs. Therefore $E_{n} \cap (H_{n, 1} \cup H_{n, 2} \cup H_{n, 3})^{c} \subseteq H_{n, 4}$ holds, i.e.
$$ E_{n} \subseteq H_{n, 1} \cup H_{n, 2} \cup H_{n, 3} \cup H_{n, 4},$$
and from (\ref{e:est2})--(\ref{e:est10-1}) we obtain
$ \lim_{n \to \infty} \Pr(E_{n})=0.$
This and relation (\ref{e:Tnfirst}) yield
\be\label{eq:Tnend}
\lim_{n \to \infty} \Pr(T_{n}> \delta)=0.
\ee
Now from (\ref{eq:AB}), (\ref{eq:Ynend}) and (\ref{eq:Tnend}) we obtain (\ref{e:max1}),
which means that $M_{n} \dto M$
in $D_{\uparrow}([0,1], \mathbb{R}^{d})$ with the weak $M_{1}$ topology, and this concludes the proof.
\end{proof}

\section{Infinite order linear processes}\label{S:InfiniteMA}

Let $(Z_{i})_{i \in \mathbb{Z}}$ be a strictly stationary sequence of regularly varying $\mathbb{R}^{d}$--valued random vectors with index $\alpha>0$, and $(C_{i})_{i \geq 0}$ a sequence of random $d \times d$ matrices
 independent of $(Z_{i})$ such that the series defining the linear process
 \begin{equation}\label{e:infLP}
 X_{i} = \sum_{j=0}^{\infty}C_{j}Z_{i-j}, \qquad i \in \mathbb{Z},
 \end{equation}
is a.s.~convergent. For $k,j \in \{1,\ldots,d\}$ let
\begin{equation*}\label{e:Cplusminus1}
D^{k,j}_{+}= \max \{ C_{i;k,j}^{+} : i \geq 0 \} \quad \textrm{and} \quad D^{k,j}_{-}= \max \{ C_{i;k,j}^{-} : i \geq 0 \},
\end{equation*}
where $C_{i;k,j}$ is the $(k,j)$th entry of the matrix $C_{i}$. Let $M_{n}$ be the partial maxima process as defined in (\ref{eq:defWn}), and $M$ the process described in $(\ref{e:limprocess})$, that is
\begin{equation*}
 M(t) = \bigg( \bigvee_{j=1}^{d} \Big(\widetilde{D}^{k,j}_{+}M^{(j+)}(t) \vee \widetilde{D}^{k,j}_{-}M^{(j-)}(t) \Big) \bigg)_{k=1,\ldots,d}, \qquad t \in [0,1],
\end{equation*}
where $M^{(j+)}$ and $M^{(j-)}$ are extremal processes with exponent measures $\nu_{j+}$ and $\nu_{j-}$ respectively, given by
$$  \nu_{j+}(\rmd x) = \mathrm{E}(Q_{1}^{(j)+})^{\alpha} \, \alpha x^{-\alpha-1}\, \rmd x
\qquad \textrm{and} \qquad  \nu_{j-}(\rmd x) =  \mathrm{E}(Q_{1}^{(j)-})^{\alpha} \, \alpha x^{-\alpha-1} \, \rmd x$$
for $x>0$ $(j=1,\ldots,d)$, with $\sum_{i}\delta_{(T_{i},P_{i}Q_{i})}$ being the Poisson process from (\ref{e:BaTa1}),
and $((\widetilde{D}^{k,j}_{+}, \widetilde{D}^{k,j}_{-})^{*}_{j =1, \ldots, d})_{k=1,\ldots,d}$ is a $2d^{2}$--dimensional random vector, independent of $(M^{(j+)}, M^{(j-)})_{j=1,\ldots,d}$, such that
$$((\widetilde{D}^{k,j}_{+}, \widetilde{D}^{k,j}_{-})^{*}_{j =1, \ldots, d})_{k=1,\ldots,d} \eind ((D^{k,j}_{+}, D^{k,j}_{-})^{*}_{j =1, \ldots, d})_{k=1,\ldots,d}.$$

In order to obtain functional convergence of the partial maxima process for infinite order linear processes, we first approximate them by a sequence of finite order linear processes, for which Theorem~\ref{t:FLT} holds, and then show that the error of approximation is negligible in the limit with respect to the weak $M_{1}$ topology. In this case, besides the conditions from Theorem~\ref{t:FLT} for finite order linear processes, we will need also some moment conditions on the sequence of coefficients. Say here that the pairwise asymptotical independence condition (\ref{e:asyind}) and Conditions~\ref{c:mixcond1} and~\ref{c:mixcond2} in the next theorem hold under strong mixing and Leadbetter's dependence condition $D'$ (\ref{e:D'cond}), as noted in Subsection~\ref{ss:PP}.

\begin{thm}\label{t:infFLT}
Let $(Z_{i})_{i \in \mathbb{Z}}$ be a strictly stationary sequence of regularly varying $\mathbb{R}^{d}$--valued random vectors with index $\alpha >0$ that satisfy $(\ref{e:asyind})$ and $(\ref{e:asyindcomp})$, and let
 $(C_{i})_{i \geq 0}$ be a sequence of random $d\times d$ matrices independent of $(Z_{i})$. Assume Conditions~\ref{c:mixcond1} and~\ref{c:mixcond2} hold with the same sequence $(r_{n})$.  If $\alpha \in (0,1)$ suppose
 \begin{equation}\label{e:momcondr}
 \sum_{j=0}^{\infty} \mathrm{E} \|C_{j}\|^{\delta} < \infty \qquad \textrm{for some}  \ \delta \in (0, \alpha),
 \end{equation}
 and
\begin{equation}\label{e:mod1}
  \sum_{j=0}^{\infty}\mathrm{E}\|C_{j}\|^{\gamma} < \infty \qquad \textrm{for some} \ \gamma \in (\alpha, 1),
\end{equation}
 and if $\alpha \geq 1$ suppose
\begin{equation}\label{eq:infmaTK3}
\sum_{j=0}^{\infty} \mathrm{E}\|C_{j}\| < \infty.
\end{equation}
Suppose condition $(\ref{e:momcondr})$ holds also when $\alpha=1$.
Then, as $n \to \infty$,
$$ M_{n}(\,\cdot\,) \dto  M(\,\cdot\,)$$
in $D_{\uparrow}([0,1], \mathbb{R}^{d})$ endowed with the weak $M_{1}$ topology.
\end{thm}

\begin{proof}
For $m \in \mathbb{N}$, $m \geq 2$, define
$$ X_{i}^{m} = \sum_{j=0}^{m-2}C_{j}Z_{i-j} + C^{(m,\vee)} Z_{i-m+1} + C^{(m, \wedge)} Z_{i-m}, \qquad i \in \mathbb{Z},$$
and
\begin{equation*}
M_{n,m}(t) = \left\{ \begin{array}{cc}
                                   \displaystyle \frac{1}{a_{n}} \bigvee_{i=1}^{\floor {nt}}X_{i}^{m} , & \quad  \displaystyle t  \in \Big[\frac{1}{n},1\Big],\\[1.7em]
                                   \displaystyle \frac{X_{1}^{m}}{a_{n}} , & \quad \displaystyle  t \in \Big[0, \frac{1}{n}\Big),
                                 \end{array}\right.
\end{equation*}
where $C^{(m, \vee)}= \max\{ C_{i} : i \geq m-1\}$ and $C^{(m, \wedge)}= \min\{ C_{i} : i \geq m-1\}$, with the maximum and minimum of matrices interpreted componentwise, i.e.~the $(k,j)$th entry of the matrix $C^{(m, \vee)}$ is $C^{(m, \vee)}_{k,j} = \max \{C_{i;k,j} : i \geq m-1\}$, and the $(k,j)$th entry of the matrix $C^{(m, \wedge)}$ is $C^{(m, \wedge)}_{k,j} = \min \{C_{i;k,j} : i \geq m-1\}$.
For $k,j \in \{1,\ldots,d\}$ define
\begin{equation*}
D^{m,k,j}_{+}= \bigg( \bigvee_{i=0}^{m-2} C_{i;k,j}^{+} \bigg) \vee C^{(m, \vee)+}_{k,j} \vee C^{(m, \wedge)+}_{k,j}
\end{equation*}
and
\begin{equation*}
D^{m,k,j}_{-}= \Big( \bigvee_{i=0}^{m-2} C_{i;k,j}^{-} \Big) \vee C^{(m, \vee)-}_{k,j} \vee C^{(m, \wedge)-}_{k,j}.
\end{equation*}
Since
$$ C^{(m, \vee)+}_{k,j} = \bigg( \bigvee_{i=m-1}^{\infty}C_{i;k,j} \bigg)^{+} = \bigvee_{i=m-1}^{\infty}C_{i;k,j}^{+}, \qquad C^{(m, \wedge)+}_{k,j} \leq \bigvee_{i=m-1}^{\infty}C_{i;k,j}^{+},$$
and
$$ C^{(m, \wedge)-}_{k,j} = \bigg( \bigwedge_{i=m-1}^{\infty}C_{i;k,j} \bigg)^{-} = \bigvee_{i=m-1}^{\infty}C_{i;k,j}^{-}, \qquad C^{(m, \vee)-}_{k,j} \leq \bigvee_{i=m-1}^{\infty}C_{i;k,j}^{-},$$
it holds that
$$ D^{m,k,j}_{+} = D^{k,j}_{+} \qquad \textrm{and} \qquad D^{m,k,j}_{-} = D^{k,j}_{-},$$
and therefore for the sequence of finite order linear processes $(X_{i}^{m})_{i}$ by Theorem~\ref{t:FLT} we obtain
$$ M_{n, m}(\,\cdot\,)  \dto M(\,\cdot\,) \qquad \textrm{as} \ n \to \infty,$$
in $D_{\uparrow}([0,1], \mathbb{R}^{d})$ with the weak $M_{1}$ topology.

If we show that for every $\epsilon >0$
\begin{equation}\label{e:Slutskyinf01}
 \lim_{m \to \infty} \limsup_{n \to \infty}\Pr[d_{p}(M_{n}, M_{n,  m})> \epsilon]=0,
\end{equation}
then by a generalization of Slutsky's theorem (see Theorem 3.5 in Resnick~\cite{Resnick07}) it will follow $M_{n} \dto M$ in $D_{\uparrow}([0,1], \mathbb{R}^{d})$ with the weak $M_{1}$ topology. According to the definition of the metric $d_{p}$ in (\ref{e:defdp}) it is enough to show
\begin{equation*}
 \lim_{m \to \infty} \limsup_{n \to \infty}\Pr[d_{M_{1}}(M_{n}^{(j)}, M_{n,m}^{(j)})> \epsilon]=0 \qquad \textrm{for every} \ j=1,\ldots,d,
\end{equation*}
and further, as in the proof of Theorem~\ref{t:FLT}, it is enough to show the last relation only for $j=1$.
Since the metric $d_{M_{1}}$ on $D_{\uparrow}([0,1], \mathbb{R})$ is bounded above by the uniform metric on $D_{\uparrow}([0,1], \mathbb{R})$, it suffices to show that
$$ \lim_{m \to \infty} \limsup_{n \to \infty}\Pr \bigg( \sup_{0 \leq t \leq 1}|M_{n}^{(1)}(t) - M_{n,m}^{(1)}(t)|> \epsilon \bigg)=0.$$
Now we treat separately the cases $\alpha \in (0,1)$ and $\alpha \in [1,\infty)$.\\[-0.2em]

Case $\alpha \in (0,1)$.
Recalling the definitions, we have
\begin{eqnarray}\label{e:unifm}
 \nonumber \Pr \bigg( \sup_{0 \leq t \leq 1}|M_{n}^{(1)}(t) - M_{n,m}^{(1)}(t)|> \epsilon \bigg) &\leq& \Pr \bigg( \bigvee_{i=1}^{n}\frac{|X_{i}^{(1)}-X_{i}^{m(1)}|}{a_{n}} > \epsilon \bigg)\\[0.5em]
  &\leq& \Pr \bigg( \sum_{i=1}^{n}\frac{|X_{i}^{(1)}-X_{i}^{m(1)}|}{a_{n}} > \epsilon \bigg).
\end{eqnarray}
Since
\begin{eqnarray}\label{e:alphafc}
\nonumber X_{i}^{(1)}-X_{i}^{m(1)} &=& \sum_{k=0}^{\infty} \sum_{j=1}^{d}C_{k;1,j}Z_{i-k}^{(j)} - \sum_{k=0}^{m-2} \sum_{j=1}^{d} C_{k;1,j}Z_{i-k}^{(j)} - \sum_{j=1}^{d} C_{1,j}^{(m,\vee)}Z_{i-m+1}^{(j)}\\[0.2em]
\nonumber && \hspace*{0.1em} - \sum_{j=1}^{d} C_{1,j}^{(m,\wedge)}Z_{i-m}^{(j)}\\[0.2em]
\nonumber &=& \sum_{k=m-1}^{\infty} \sum_{j=1}^{d}C_{k;1,j}Z_{i-k}^{(j)} - \sum_{j=1}^{d} C_{1,j}^{(m,\vee)}Z_{i-m+1}^{(j)} - \sum_{j=1}^{d} C_{1,j}^{(m,\wedge)}Z_{i-m}^{(j)}\\[0.2em]
 \nonumber &=& \sum_{j=1}^{d} \bigg(  \sum_{k=m+1}^{\infty} C_{k;1,j}Z_{i-k}^{(j)} + ( C_{m-1;1,j} - C_{1,j}^{(m,\vee)} ) Z_{i-m+1}^{(j)}\\[0.2em]
  && \hspace*{0.1em} + ( C_{m;1,j} - C_{1,j}^{(m,\wedge)} ) Z_{i-m}^{(j)}  \bigg),
\end{eqnarray}
and
$$ | C_{m-1;1,j} - C_{1,j}^{(m,\vee)} | \leq \sum_{l=m-1}^{\infty}|C_{l;1,j}|, \qquad | C_{m;1,j} - C_{1,j}^{(m,\wedge)} | \leq \sum_{l=m-1}^{\infty}|C_{l;1,j}|,$$
we have
\begin{eqnarray*}
  \sum_{i=1}^{n}|X_{i}^{(1)}-X_{i}^{m(1)}| & &\\[0.1em]
   & \hspace{-10em} \leq & \hspace*{-5em} \sum_{i=1}^{n} \sum_{j=1}^{d} \bigg(\sum_{k=m+1}^{\infty}|C_{k;1,j}|\,|Z_{i-k}^{(j)}| + \sum_{l=m-1}^{\infty}|C_{l;1,j}| ( |Z_{i-m+1}^{(j)}| + |Z_{i-m}^{(j)}|) \bigg)\\[0.1em]
    & \hspace{-10em} \leq & \hspace*{-5em} \sum_{j=1}^{d} \bigg[ \sum_{i=-\infty}^{0}|Z_{i-m}^{(j)}| \sum_{s=1}^{n}|C_{m-i+s;1,j}| + \bigg( 2 \sum_{l=m-1}^{\infty}|C_{l;1,j}| \bigg) \sum_{i=1}^{n+1} |Z_{i-m}^{(j)}| \bigg]\\[0.1em]
    & \hspace{-10em} \leq & \hspace*{-5em}  \sum_{j=1}^{d} \bigg[ \sum_{i=-\infty}^{0}|Z_{i-m}^{(j)}| \sum_{s=1}^{n}\|C_{m-i+s}\| + \bigg( 2 \sum_{l=m-1}^{\infty}\|C_{l}\| \bigg) \sum_{i=1}^{n+1} |Z_{i-m}^{(j)}| \bigg],\\[0.1em]
\end{eqnarray*}
where in the second inequality above we used a change of variables and rearrangement of sums. Therefore
 \begin{eqnarray*}
 \Pr \bigg( \sum_{i=1}^{n}\frac{|X_{i}^{(1)}-X_{i}^{m(1)}|}{a_{n}} > \epsilon \bigg) &&\\[0.2em]
  &\hspace*{-18em} \leq & \hspace*{-9em} \sum_{j=1}^{d} \Pr \bigg( \sum_{i=-\infty}^{0}|Z_{i-m}^{(j)}| \sum_{s=1}^{n}\|C_{m-i+s}\| + \bigg( 2 \sum_{l=m-1}^{\infty}\|C_{l}\| \bigg) \sum_{i=1}^{n+1} |Z_{i-m}^{(j)}| > \frac{\epsilon}{d} \bigg).
 \end{eqnarray*}
Conditions $(\ref{e:momcondr})$ and $(\ref{e:mod1})$ by Lemma~\ref{l:ALEA1} in Appendix imply
$$ \lim_{m \to \infty} \limsup_{n \to \infty}  \Pr \bigg( \sum_{i=-\infty}^{0}|Z_{i-m}^{(j)}| \sum_{s=1}^{n}\|C_{m-i+s}\| + \bigg( 2 \sum_{l=m-1}^{\infty}\|C_{l}\| \bigg) \sum_{i=1}^{n+1} |Z_{i-m}^{(j)}| > \frac{\epsilon}{d} \bigg) =0$$
for every $j=1,\ldots,d$, and hence from (\ref{e:unifm}) we obtain
$$ \lim_{m \to \infty} \limsup_{n \to \infty}  \Pr \bigg( \sup_{0 \leq t \leq 1}|M_{n}^{(1)}(t) - M_{n, m}^{(1)}(t)|> \epsilon \bigg)= 0,$$
which means that $M_{n} \dto M$ as $n \to \infty$ in $D_{\uparrow}([0,1], \mathbb{R}^{d})$ with the weak $M_{1}$ topology.\\[-0.7em]

Case $\alpha \in [1,\infty)$.
Define $$ A_{k,j} = \left\{ \begin{array}{cl}
                                   C_{k;1,j}-C_{1,j}^{(m,\vee)}, & \quad \textrm{if} \ k=m-1,\\[0.5em]
                                   C_{k;1,j}-C_{1,j}^{(m,\wedge)}, & \quad \textrm{if} \ k=m,\\[0.5em]
                                   C_{k;1,j}, & \quad \textrm{if} \ k \geq m+1,
                                 \end{array}\right.$$
for $k \geq m-1$ and $j \in \{1,\ldots,d\}$, and note that by (\ref{e:alphafc}), for $t \in [0,1]$ we have
\begin{eqnarray*}
\nonumber  |M_{n}^{(1)}(t) - M_{n,m}^{(1)}(t)| & \leq &  \bigvee_{i=1}^{\floor{nt}} \frac{|X_{i}^{(1)}-X_{i}^{m(1)}|}{a_{n}}  \\[0.1em]
& \hspace*{-17em} = & \hspace*{-8.5em} \bigvee_{i=1}^{\floor{nt}}  \sum_{j=1}^{d} \bigg|  \sum_{k=m+1}^{\infty} C_{k;1,j}\frac{Z_{i-k}^{(j)}}{a_{n}} + \Big( C_{m-1;1,j} - C_{1,j}^{(m,\vee)} \Big) \frac{Z_{i-m+1}^{(j)}}{a_{n}} + \Big( C_{m;1,j} - C_{1,j}^{(m,\wedge)} \Big) \frac{Z_{i-m}^{(j)}}{a_{n}}  \bigg|\\[0.1em]
& \hspace*{-17em} = & \hspace*{-8.5em} \bigvee_{i=1}^{\floor{nt}}  \sum_{j=1}^{d} \bigg|  \sum_{k=m-1}^{\infty} A_{k,j}\frac{Z_{i-k}^{(j)}}{a_{n}}   \bigg|.
\end{eqnarray*}
Therefore
$$\Pr \bigg( \sup_{0 \leq t \leq 1}|M_{n}^{(1)}(t) - M_{n,m}^{(1)}(t)|> \epsilon \bigg) \leq \Pr \bigg( \bigvee_{i=1}^{n}  \sum_{j=1}^{d} \bigg|  \sum_{k=m-1}^{\infty} \frac{A_{k,j}Z_{i-k}^{(j)}}{a_{n}}  \bigg| > \epsilon \bigg),$$
and Lemma~\ref{l:JTSA1} in Appendix yields
$$ \lim_{m \to \infty} \limsup_{n \to \infty}  \Pr \bigg( \sup_{0 \leq t \leq 1}|M_{n}^{(1)}(t) - M_{n, m}^{(1)}(t)|> \epsilon \bigg)= 0.$$
We conclude that in this case also $M_{n} \dto M$ as $n \to \infty$ in $D_{\uparrow}([0,1], \mathbb{R}^{d})$ with the weak $M_{1}$ topology.
\end{proof}

\begin{rem}\label{r:determc}
When the sequence of coefficients $(C_{i})$ is deterministic, the limiting process $M$ in Theorem~\ref{t:infFLT} has the following representation
$$M(t) = \bigvee_{T_{i} \leq t}P_{i}S_{i}, \qquad t \in [0,1],$$
where $S_{i}=(S_{i}^{(1)}, \ldots, S_{i}^{(d)})$, with
$$ S_{i}^{(k)} = \bigvee_{j=1}^{d}(D_{+}^{k,j}Q_{i}^{(j)+} \vee D_{-}^{k,j}Q_{i}^{(j)-}) \qquad \textrm{for} \ k=1,\ldots,d.$$
By Propositions 5.2 and 5.3 in Resnick~\cite{Resnick07} the point process $\sum_{i}\delta_{(T_{i}, P_{i}S_{i})}$ is a Poisson process with intensity measure $Leb \times \rho$, where for $x \in [0,\infty)^{d}$, $x \neq 0$,
\begin{equation*}
\rho([[0,x]]^{c}) = \int_{0}^{\infty} \Pr \bigg(y \bigvee_{k=1}^{d}\frac{S_{1}^{(k)}}{x^{(k)}} >1 \bigg)\,\alpha y^{-\alpha-1}\,\rmd y.
\end{equation*}
Therefore $M$ is an extremal process with exponent measure $\rho$. Similarly, the components $M^{(k)}$ of $M$ are extremal processes with exponent measures $\rho_{k}$ respectively, where for $z>0$
\begin{equation*}
\rho_{k}((z,\infty)) = \int_{0}^{\infty} \Pr (y S_{1}^{(k)} > z )\,\alpha y^{-\alpha-1}\,\rmd y = \mathrm{E}(S_{1}^{(k)})^{\alpha}z^{-\alpha},
\end{equation*}
where the last equality holds by a change of variables and Fubini's theorem. Therefore
$\rho_{k}(\rmd y)=\mathrm{E}(S_{1}^{(k)})^{\alpha}\alpha y^{-\alpha-1} \rmd y$ for $y>0$. For deterministic coefficients condition (\ref{e:mod1}) can be dropped since it is implied by (\ref{e:momcondr}) (see Krizmani\'{c}~\cite{Kr22-2}, Remark 3.4).
\end{rem}

\begin{rem}
A special case of multivariate linear processes studied in this paper is
$$ X_{i} = \sum_{j=0}^{\infty}B_{j}Z_{i-j}, \qquad i \in \mathbb{Z},$$
 where $(Z_{i})_{i \in \mathbb{Z}}$ is a strictly stationary sequence of regularly varying $\mathbb{R}^{d}$--valued random vectors with index $\alpha>0$, and $(B_{i})_{i \geq 0}$ is a sequence of random variables
 independent of $(Z_{i})$. To obtain this linear process from the general one in (\ref{e:infLP}) take
 $$C_{i;k,j}= \left\{ \begin{array}{cl}
  B_{i}, & \quad \textrm{if} \ k=j,\\[0.4em]
  0, & \quad \textrm{if} \ k \neq j,
\end{array}\right.$$
for $i \geq 0$ and $k,j \in \{1,\ldots,d\}$. In this case the limiting process $M$ in Theorem~\ref{t:infFLT} reduces to
$$M(t) = \Big(\widetilde{D}^{k,k}_{+}M^{(k+)}(t) \vee \widetilde{D}^{k,k}_{-}M^{(k-)}(t) \Big)_{k=1,\ldots,d} =  \Big(\widetilde{B}_{+}M^{(k+)}(t) \vee \widetilde{B}_{-}M^{(k-)}(t) \Big)_{k=1,\ldots,d}$$
for $t \in [0,1]$, where $(\widetilde{B}_{+}, \widetilde{B}_{-})$ is a two dimensional random vector, independent of $(M^{(k+)}, M^{(k-)})^{*}_{k=1,\ldots,d}$, such that $(\widetilde{B}_{+}, \widetilde{B}_{-}) \eind (\bigvee_{i \geq 0}B_{i}^{+}, \bigvee_{i \geq 0}B_{i}^{-})$. By an application of Propositions 5.2 and 5.3 in Resnick~\cite{Resnick07} we can represent $M$ in the form $M(t)=\widetilde{B}_{+}M_{+}(t) \vee \widetilde{B}_{-}M_{-}(t)$ for $t \in [0,1]$, where $M_{+}(t)=(M^{(k+)})_{k=1,\ldots,d}$ and $M_{-}(t)=(M^{(k-)})_{k=1,\ldots,d}$ are extremal processes with exponent measures $\nu_{+}$ and $\nu_{-}$ respectively, where for $x \in [0,\infty)^{d}$, $x \neq 0$,
$$\nu_{+}([[0,x]]^{c}) = \int_{0}^{\infty} \Pr \bigg(y \bigvee_{k=1}^{d}\frac{Q_{1}^{(k)+}}{x^{(k)}} > 1 \bigg)\,\alpha y^{-\alpha-1}\,\rmd y$$
and
$$\nu_{-}([[0,x]]^{c}) = \int_{0}^{\infty} \Pr \bigg(y \bigvee_{k=1}^{d}\frac{Q_{1}^{(k)-}}{x^{(k)}} > 1 \bigg)\,\alpha y^{-\alpha-1}\,\rmd y.$$
Taking into account the considerations in Remark~\ref{r:determc} we see that in the univariate case (when $d=1$) the measures $\nu_{+}$ and $\nu_{-}$ are of the form
$$ \nu_{+}(\rmd y)=\mathrm{E}(Q_{1}^{+})^{\alpha}\alpha y^{-\alpha-1} \rmd y \qquad \textrm{and} \qquad \nu_{-}(\rmd y)=\mathrm{E}(Q_{1}^{-})^{\alpha}\alpha y^{-\alpha-1} \rmd y$$
for $y>0$, with
$$ \mathrm{E}(Q_{1}^{+})^{\alpha} = \Pr(Q_{1} >0) = \lim_{x \to \infty}\frac{\Pr(Z_{1}>x)}{\Pr(|Z_{1}|>x)}$$
and
$$ \mathrm{E}(Q_{1}^{-})^{\alpha} = \Pr(Q_{1} <0) = \lim_{x \to \infty}\frac{\Pr(Z_{1}<-x)}{\Pr(|Z_{1}|>x)},$$
which follows from the definitions of $Q_{1}$ and the tail process of the sequence $(Z_{i})$. This is the same form of the limiting process $M$ as in the i.i.d.~case obtained in Krizmani\'{c}~\cite{Kr22-2}.
\end{rem}

In the following example we show that the functional convergence in the weak $M_{1}$ topology in Theorems~\ref{t:FLT} and~\ref{t:infFLT} in general can not be replaced by convergence in the stronger standard $M_{1}$ topology.

\begin{exmp}
Let $(T_i)_{i \in \mathbb{Z}}$ be a sequence of i.i.d.~unit Fr\'{e}chet random variables, i.e.~$\Pr(T_{i} \leq x) = e^{-1/x}$ for $x>0$. Hence $T_{i}$ is non-negative and regularly varying with index $\alpha=1$. Take a sequence of positive real numbers $(a_{n})$ such that
$n \Pr (T_{1}>a_{n}) \to 1/2$ as $n \to \infty$. Let
$$
Z_i = (T_{2i-1}, T_{2i}), \qquad {i \in \mathbb{Z}}.
$$
Stationarity and independence in the sequence $(T_{i})$ yield
$$ n \Pr(\|Z_{1}\| > a_{n}) = 2n \Pr(T_{1}> a_{n}) - n [\Pr(T_{1} > a_{n})]^{2},$$
and this implies
$n \Pr (\|Z_{1}\|>a_{n}) \to 1$ as $n \to \infty$.
Proposition~\ref{p:jrvai} implies that every $Z_{i}$ is regularly varying with index $\alpha=1$, and also that the random process $(Z_{i})_{i \in \mathbb{Z}}$ is jointly regularly varying. Since $(Z_{i})$ is a sequence of i.i.d.~regularly varying random vectors, it follows immediately that relation (\ref{e:asyind}) and Conditions~\ref{c:mixcond1} and~\ref{c:mixcond2} hold. Relation (\ref{e:asyindcomp}) holds since the components of each $Z_{i}$ are independent and regularly varying. Therefore $(Z_{i})$ satisfies all the conditions of Theorem~\ref{t:FLT}, and the partial maxima process $M_{n}(\,\cdot\,)$ of the linear process
$$X_{i}=C_{0}Z_{i} + C_{1}Z_{i-1}, \qquad i \in \mathbb{Z},$$
with
$$ C_{0}=\left(
           \begin{array}{cc}
             1 & 1 \\
             0 & 0 \\
           \end{array}
         \right) \qquad \textrm{and} \qquad C_{1}=\left(
           \begin{array}{cc}
             0 & 0 \\
             1 & 1 \\
           \end{array}
         \right),$$
converges in distribution in $D_{\uparrow}([0,1], \mathbb{R}^{2})$ with the weak $M_{1}$ topology.

Next we show that $M_{n}(\,\cdot\,)$ does not converge in distribution under the standard $M_{1}$ topology on $D_{\uparrow}([0,1], \mathbb{R}^{2})$. This shows that the weak $M_{1}$ topology in Theorems~\ref{t:FLT} and~\ref{t:infFLT} in general cannot be replaced by the standard $M_{1}$ topology.
Let
 $$ V_{n}(t) = M_{n}^{(1)}(t) - M_{n}^{(2)}(t), \qquad t \in [0,1],$$
where
$$ M_{n}^{(1)}(t) = \bigvee_{i=1}^{\floor{nt}} \frac{Z_{i}^{(1)}+Z_{i}^{(2)}}{a_{n}}  = \bigvee_{i=1}^{\floor{nt}} \frac{T_{2i-1}+T_{2i}}{a_{n}}$$
and
 $$ M_{n}^{(2)}(t) = \bigvee_{i=1}^{\floor{nt}} \frac{Z_{i-1}^{(1)}+Z_{i-1}^{(2)}}{a_{n}} =\bigvee_{i=1}^{\floor{nt}} \frac{T_{2i-3}+T_{2i-2}}{a_{n}} .$$
The first step is to show that $V_{n}(\,\cdot\,)$ does not converge in distribution in $D([0,1], \mathbb{R})$ endowed with the standard $M_{1}$ topology. According to Skorokhod~\cite{Sk56} (see also Avram and Taqqu~\cite{AvTa92}, Proposition 2) it suffices to show that
 \begin{equation}\label{e:osc1}
 \lim_{\delta \to 0} \limsup_{n \to \infty} \Pr ( \omega_{\delta}(V_{n}) > \epsilon ) > 0
 \end{equation}
 for some $\epsilon >0$, where
 $$ \omega_{\delta}(x) = \sup_{{\footnotesize \begin{array}{c}
                                t_{1} \leq t \leq t_{2} \\
                                0 \leq t_{2}-t_{1} \leq \delta
                              \end{array}}
} M(x(t_{1}), x(t), x(t_{2}))$$
($x \in D([0,1], \mathbb{R}), \delta >0)$ and
$$ M(x_{1},x_{2},x_{3}) = \left\{ \begin{array}{ll}
                                   0, & \ \ \textrm{if} \ x_{2} \in [x_{1}, x_{3}], \\[0.3em]
                                   \min\{ |x_{2}-x_{1}|, |x_{3}-x_{2}| \}, & \ \ \textrm{otherwise}.
                                 \end{array}\right.$$
Denote by $i'=i'(n)$ the index at which $\max_{1 \leq i \leq n-1}T_{i}$ is obtained. Fix $\epsilon >0$ and let
 $$A_{n,\epsilon} = \{ T_{i'} > \epsilon a_{n} \} = \Big\{ \max_{1 \leq i \leq n-1}T_{i} > \epsilon a_{n}\Big\}$$
 and
 $$ B_{n,\epsilon} = \{T_{i'}>\epsilon a_{n} \ \textrm{and} \ \exists\,k \in \{-i'-1,\ldots,3\} \setminus \{0\}
 \ \textrm{such that} \ T_{i'+k} > \epsilon a_{n} / 8 \}.$$
 From the regular variation property of $T_{1}$ we obtain $\lim_{n \to \infty} n
 \Pr(T_{1}> c a_{n}) = (2c)^{-1}$ for $c>0$, and this together with the fact that $(T_{i})$ is an i.i.d.~sequence yield
 \begin{equation}\label{e:limAn}
  \lim_{n \to \infty}\Pr(A_{n,\epsilon}) = 1- \lim_{n \to \infty} \bigg( 1- \frac{n\Pr(T_{1}>\epsilon a_{n})}{n} \bigg)^{n-1} = 1 - e^{-(2\epsilon)^{-1}}
  \end{equation}
and
 \begin{eqnarray}\label{e:limBn}
 \nonumber \limsup_{n \to \infty} \Pr(B_{n,\epsilon})  &\leq&  \limsup_{n \to \infty} \sum_{i=1}^{n-1} \sum_{\footnotesize \begin{array}{c}
                                k=-n \\
                                k \neq 0
                              \end{array}}^{3}  \Pr(T_{i}>\epsilon a_{n}, T_{i+k} > \epsilon a_{n}/8)\\[0.4em]
 \nonumber & \leq &  \limsup_{n \to \infty}\,(n-1)(n+3) \Pr(T_{1}> \epsilon a_{n}) \Pr(T_{1} > \epsilon a_{n}/8)\\[0.4em]
 &=& 2 \epsilon^{-2}.
 \end{eqnarray}
Note that on the event $A_{n,\epsilon} \setminus B_{n,\epsilon}$ it holds that $T_{i'}> \epsilon a_{n}$ and $T_{i'+k} \leq \epsilon a_{n}/8$ for every $k \in \{-i'-1,\ldots,3\} \setminus \{0\}$.
Now we distinguish two cases.
\begin{itemize}
  \item[(i)] $i'$ is an even number. Then $i'=2i^{*}$ for some integer $i^{*}$. Observe that on the set $A_{n,\epsilon} \setminus B_{n,\epsilon}$ we have
  $$ M_{n}^{(1)} \Big( \frac{i^{*}}{n} \Big) = \frac{T_{i'-1}+T_{i'}}{a_{n}} > \epsilon \quad \textrm{and} \quad  M_{n}^{(2)} \Big( \frac{i^{*}}{n} \Big) = \bigvee_{i=1}^{i^{*}} \frac{T_{2i-3}+T_{2i-2}}{a_{n}} \leq \frac{\epsilon}{4},$$
  and similarly
   $$M_{n}^{(1)} \Big( \frac{i^{*}-1}{n} \Big) \leq \frac{\epsilon}{4} \quad \textrm{and} \quad  M_{n}^{(2)} \Big( \frac{i^{*}-1}{n} \Big)  \leq \frac{\epsilon}{4}.$$ This implies
   $$V_{n} \Big( \frac{i^{*}}{n} \Big) = M_{n}^{(1)} \Big( \frac{i^{*}}{n} \Big) - M_{n}^{(2)} \Big( \frac{i^{*}}{n} \Big) > \frac{3\epsilon}{4},$$
   and
   $$V_{n} \Big( \frac{i^{*}-1}{n} \Big) = M_{n}^{(1)} \Big( \frac{i^{*}-1}{n} \Big) - M_{n}^{(2)} \Big( \frac{i^{*}-1}{n} \Big) \in \Big[ -\frac{\epsilon}{4},  \frac{\epsilon}{4} \Big].$$
   Further, on the set $A_{n,\epsilon} \setminus B_{n,\epsilon}$ it holds that
   $$ M_{n}^{(1)} \Big( \frac{i^{*}+1}{n} \Big) = \frac{T_{i'-1}+T_{i'}}{a_{n}}  \quad \textrm{and} \quad  M_{n}^{(2)} \Big( \frac{i^{*}+1}{n} \Big) = \frac{T_{i'-1}+T_{i'}}{a_{n}},$$
   which yields
   $$ V_{n}\Big( \frac{i^{*}+1}{n} \Big) = 0.$$
   \item[(ii)] $i'$ is an odd number. Then $i'=2i^{*}-1$ for some integer $i^{*}$. Similarly as in (i)
   on the event $A_{n,\epsilon} \setminus B_{n,\epsilon}$ one obtains
$$ V_{n} \Big( \frac{i^{*}}{n} \Big)  > \frac{3 \epsilon}{4}, \quad V_{n} \Big( \frac{i^{*}-1}{n} \Big) \in \Big[ - \frac{\epsilon}{4}, \frac{\epsilon}{4} \Big] \quad \textrm{and} \quad V_{n} \Big( \frac{i^{*}+1}{n} \Big) = 0.$$
\end{itemize}
Hence on the set $A_{n,\epsilon} \setminus B_{n,\epsilon}$ we have
\begin{equation}\label{e:inc1}
  \Big| V_{n} \Big( \frac{i^{*}}{n} \Big) - V_{n} \Big( \frac{i^{*}-1}{n} \Big) \Big|  > \frac{3 \epsilon}{4} - \frac{\epsilon}{4} = \frac{\epsilon}{2}
\end{equation}
and
\begin{equation}\label{e:inc2}
  \Big| V_{n} \Big( \frac{i^{*}+1}{n} \Big) - V_{n} \Big( \frac{i^{*}}{n} \Big) \Big| > \frac{3 \epsilon}{4}.
\end{equation}
Note that on the set $A_{n,\epsilon} \setminus B_{n,\epsilon}$ one also has
$$ V_{n} \Big( \frac{i^{*}}{n} \Big) \notin \Big[ V_{n} \Big( \frac{i^{*}-1}{n} \Big), V_{n} \Big( \frac{i^{*}+1}{n} \Big) \Big],$$
and therefore taking into account (\ref{e:inc1}) and (\ref{e:inc2}) we obtain
\begin{eqnarray*}
  \omega_{2/n}(V_{n}) & = & \sup_{{\footnotesize \begin{array}{c}
                                t_{1} \leq t \leq t_{2} \\
                                0 \leq t_{2}-t_{1} \leq 2/n
                              \end{array}}
} M(V_{n}(t_{1}), V_{n}(t), V_{n}(t_{2})) \\[0.8em]
   & \geq & M \Big( V_{n} \Big( \frac{i^{*}-1}{n} \Big), V_{n} \Big( \frac{i^{*}}{n} \Big), V_{n} \Big( \frac{i^{*}+1}{n} \Big) \Big) > \frac{\epsilon}{2}
\end{eqnarray*}
on the event $A_{n,\epsilon} \setminus B_{n,\epsilon}$. Therefore, since $\omega_{\delta}(\,\cdot\,)$ is nondecreasing in $\delta$, it holds that
 \begin{eqnarray}\label{e:oscM1}
  \nonumber \liminf_{n \to \infty} \Pr(A_{n,\epsilon} \setminus B_{n,\epsilon}) & \leq & \liminf_{n \to \infty}
 \Pr ( \omega_{2/n} (V_{n}) >  \epsilon /2 )\\[0.4em]
 & \leq &   \lim_{\delta \to 0} \limsup_{n \to \infty}  \Pr ( \omega_{\delta} (V_{n}) >  \epsilon/2 ).
 \end{eqnarray}
Since $x^{2}(1-e^{-(2x)^{-1}})$ tends to infinity as $x \to \infty$, we can find $\epsilon >0$ such that  $\epsilon^{2}(1-e^{-(2\epsilon)^{-1}}) > 2$, that is $ 1-e^{-(2\epsilon)^{-1}} > 2 \epsilon^{-2}$.
 For this $\epsilon$, by relations (\ref{e:limAn}) and (\ref{e:limBn}), we have
 $$\lim_{n \to \infty} \Pr(A_{n,\epsilon}) > \limsup_{n \to \infty} \Pr(B_{n,\epsilon}),$$
 i.e.
 $$ \liminf_{n \to \infty} \Pr(A_{n,\epsilon} \setminus B_{n,\epsilon}) \geq \lim_{n \to \infty}\Pr(A_{n,\epsilon}) - \limsup_{n \to \infty} \Pr(B_{n,\epsilon}) >0.$$
Therefore by (\ref{e:oscM1}) we obtain
$$ \lim_{\delta \to 0} \limsup_{n \to \infty}  \Pr ( \omega_{\delta} (V_{n}) >  \epsilon/2) > 0$$
and (\ref{e:osc1}) holds, and $V_{n}(\,\cdot\,)$ does not converge in distribution in $D([0,1], \mathbb{R})$ with the standard $M_{1}$ topology.

 To finish, if $M_{n}(\,\cdot\,)$ would converge in distribution in the standard $M_{1}$ topology on $D_{\uparrow}([0,1], \mathbb{R}^{2})$, and then also on $D([0,1], \mathbb{R}^{2})$, using the fact that linear combinations of the coordinates are continuous in the same topology (see Whitt~\cite{Whitt02}, Theorems 12.7.1 and 12.7.2) and the continuous mapping theorem, we would obtain that $V_{n}(\,\cdot\,) = M_{n}^{(1)}(\,\cdot\,) - M_{n}^{(1)}(\,\cdot\,)$ converges in $D([0,1], \mathbb{R})$ with the standard $M_{1}$ topology, which is impossible, as is shown above.
\end{exmp}

\section{Appendix}

We provide some technical results used in the proof of Theorem~\ref{t:infFLT}. These results are based on the corresponding results in Krizmani\'{c}~\cite{Kr19}, \cite{Kr22-2} for functional convergence of partial sums and maxima of univariate linear processes with random coefficients and heavy-tailed innovations.

Let $(Z_{i})_{i \in \mathbb{Z}}$ be a strictly stationary sequence of regularly varying $\mathbb{R}^{d}$--valued random vectors with index $\alpha>0$, and $(C_{i})_{i \geq 0}$ a sequence of random $d \times d$ matrices
 independent of $(Z_{i})$. Let $(a_{n})$ be a sequence of positive real numbers satisfying $(\ref{eq:niz})$.

\begin{lem}\label{l:ALEA1}
Let $\alpha \in (0,1)$, and assume Conditions $(\ref{e:momcondr})$ and $(\ref{e:mod1})$ hold. Then for every $j \in \{1,\ldots,d\}$ and $\epsilon >0$,
$$ \lim_{m \to \infty} \limsup_{n \to \infty} \Pr \bigg[ \sum_{i=-\infty}^{0}\frac{|Z_{i-m}^{(j)}|}{a_{n}} \sum_{s=1}^{n}\|C_{m-i+s}\| + \bigg( 2 \sum_{l=m-1}^{\infty}\|C_{l}\| \bigg) \sum_{i=1}^{n+1}\frac{|Z_{i-m}^{(j)}|}{a_{n}} > \epsilon \bigg] =0.$$
\end{lem}
\begin{proof}
Let
$$  A^{n,m}_{i} =  \left\{ \begin{array}{cl}
                                   \displaystyle \sum_{s=1}^{n} \|C_{m-i+s}\|, & \quad  i \leq 0,\\[1em]
                                   \displaystyle 2 \sum_{l=m-1}^{\infty}\|C_{l}\|, & \quad i=1,\ldots,n+1.
                                 \end{array}\right.$$
We have to show
\be\label{eq:Oiqn1}
\lim_{m \to \infty} \limsup_{n \to \infty} \Pr \bigg(\sum_{i=-\infty}^{n+1} \frac{A^{n,m}_{i}|Z_{i-m}^{(j)}|}{a_{n}} > \epsilon \bigg)=0.
\ee
Let
$$ Z^{(j)\leq}_{i,n} = \frac{Z_{i}^{(j)}}{a_{n}} 1_{\big\{ \frac{\|Z_{i}\|}{a_{n}} \leq 1 \big\}} \qquad \textrm{and} \qquad
Z^{(j)>}_{i,n} = \frac{Z_{i}^{(j)}}{a_{n}} 1_{\big\{ \frac{\|Z_{i}\|}{a_{n}} > 1 \big\}},$$
and note that the probability in (\ref{eq:Oiqn1}) is bounded above by
\begin{equation*}\label{e:Diqn1}
\Pr \bigg(\sum_{i=-\infty}^{n+1} A^{n,m}_{i} |Z^{(j)\leq}_{i-m,n}| > \frac{\epsilon}{2} \bigg) +
\Pr \bigg(\sum_{i=-\infty}^{n+1} A^{n,m}_{i} |Z^{(j)>}_{i-m,n}| > \frac{\epsilon}{2} \bigg).
\end{equation*}
Now using the same arguments from the proof of Lemma 3.1 in Krizmani\'{c}~\cite{Kr19} we obtain
\begin{equation}\label{e:Diqn2}
   \Pr \bigg(\sum_{i=-\infty}^{n+1} A^{n,m}_{i} |Z^{(j)\leq}_{i-m,n}| > \frac{\epsilon}{2} \bigg)  \leq   (2^{\gamma}+1) \Big(\frac{\epsilon}{2} \Big)^{-\gamma} (n+1) \mathrm{E}|Z^{(j)\leq}_{1,n}|^{\gamma}  \sum_{l=m-1}^{\infty} \mathrm{E} \|C_{l}\|^{\gamma}
\end{equation}
and
\begin{equation}\label{e:Diqn3}
   \Pr \bigg(\sum_{i=-\infty}^{n+1} A^{n,m}_{i} |Z^{(j)>}_{i-m,n}| > \frac{\epsilon}{2} \bigg)
   \leq (2^{\delta}+1) \Big(\frac{\epsilon}{2} \Big)^{-\delta} (n+1) \mathrm{E}|Z^{(j)>}_{1,n}|^{\delta}  \sum_{l=m-1}^{\infty} \mathrm{E} \|C_{l}\|^{\delta}.
\end{equation}
By Karamata's theorem and (\ref{eq:niz})
$$ \limsup_{n \to \infty}\,(n+1)\mathrm{E}|Z^{(j)\leq}_{1,n}|^{\gamma} \leq \limsup_{n \to \infty} \frac{\mathrm{E}(\|Z_{1}\|^{\gamma}1_{\{ \|Z_{1}\| \leq a_{n}\}})}{a_{n}^{\gamma} \Pr(\|Z_{1}\| > a_{n})}\,(n+1) \Pr(\|Z_{1}\| > a_{n}) = \frac{\alpha}{\gamma - \alpha},$$
and
$$ \limsup_{n \to \infty}\,(n+1)\mathrm{E}|Z^{(j)>}_{1,n}|^{\delta} \leq \limsup_{n \to \infty} \frac{\mathrm{E}(\|Z_{1}\|^{\delta}1_{\{ \|Z_{1}\| > a_{n}\}})}{a_{n}^{\delta} \Pr(\|Z_{1}\| > a_{n})}\,(n+1) \Pr(\|Z_{1}\| > a_{n}) = \frac{\alpha}{\alpha - \delta}.$$
From this and relations (\ref{e:Diqn2}) and (\ref{e:Diqn3}) we conclude that
$$ \limsup_{n \to \infty} \Pr \bigg(\sum_{i=-\infty}^{n+1} \frac{A^{n,m}_{i}|Z_{i-m}^{(j)}|}{a_{n}} > \epsilon \bigg) \leq A \bigg( \sum_{l=m-1}^{\infty} \mathrm{E} \|C_{l}\|^{\gamma} + \sum_{l=m-1}^{\infty} \mathrm{E} \|C_{l}\|^{\delta} \bigg),$$
for some constant $A$.
Now letting $m \to \infty$, we see that conditions (\ref{e:momcondr}) and (\ref{e:mod1}) imply (\ref{eq:Oiqn1}).
\end{proof}

\begin{lem}\label{l:JTSA1}
Let $\alpha \in [1,\infty)$, and assume Condition $(\ref{eq:infmaTK3})$ holds. If $\alpha=1$ assume also Condition $(\ref{e:momcondr})$ holds. Then for every $\epsilon >0$,
$$ \lim_{m \to \infty} \limsup_{n \to \infty} \Pr \bigg[ \bigvee_{i=1}^{n} \sum_{j=1}^{d} \bigg| \sum_{k=m-1}^{\infty}\frac{A_{k,j}Z_{i-k}^{(j)}}{a_{n}} \bigg| > \epsilon \bigg] =0,$$
where
$$ A_{k,j} = \left\{ \begin{array}{cl}
                                   C_{k;1,j}-C_{1,j}^{(m,\vee)}, & \quad \textrm{if} \ k=m-1,\\[0.5em]
                                   C_{k;1,j}-C_{1,j}^{(m,\wedge)}, & \quad \textrm{if} \ k=m,\\[0.5em]
                                   C_{k;1,j}, & \quad \textrm{if} \ k \geq m+1.
                                 \end{array}\right.$$
\end{lem}
\begin{proof}
Define $Z_{i,n}^{(j)\leq} = a_{n}^{-1} Z_{i}^{(j)} 1_{\{ \|Z_{i}\| \leq a_{n} \}}$ and $Z_{i,n}^{(j)>} = a_{n}^{-1} Z_{i}^{(j)} 1_{\{ \|Z_{i}\| > a_{n} \}}$  for $i \in \mathbb{Z}$, $n \in \mathbb{N}$ and $j \in \{1,\ldots,d\}$,
and note that
$$  \sum_{k=m-1}^{\infty}\frac{A_{k,j}Z_{i-k}^{(j)}}{a_{n}}  =  \sum_{k=m-1}^{\infty} A_{k,j}Z_{i-k,n}^{(j)\leq} + \sum_{k=m-1}^{\infty} A_{k,j}Z_{i-k,n}^{(j)>}.$$
Hence
\begin{eqnarray}
\nonumber   \Pr \bigg( \bigvee_{i=1}^{n} \sum_{j=1}^{d} \bigg| \sum_{k=m-1}^{\infty}\frac{A_{k,j}Z_{i-k}^{(j)}}{a_{n}} \bigg| > \epsilon \bigg) &\leq & \sum_{j=1}^{d} \Pr \bigg( \bigvee_{i=1}^{n} \bigg| \sum_{k=m-1}^{\infty}\frac{A_{k,j}Z_{i-k}^{(j)}}{a_{n}} \bigg| > \frac{\epsilon}{d} \bigg)\\[0.4em]
 \nonumber & \hspace*{-28em} \leq & \hspace*{-14em}  \sum_{j=1}^{d} \Pr \bigg( \bigvee_{i=1}^{n}  \bigg| \sum_{k=m-1}^{\infty} A_{k,j}Z_{i-k,n}^{(j)\leq} \bigg| > \frac{\epsilon}{2d} \bigg) +  \sum_{j=1}^{d} \Pr \bigg( \bigvee_{i=1}^{n} \bigg| \sum_{k=m-1}^{\infty} A_{k,j}Z_{i-k,n}^{(j)>} \bigg| > \frac{\epsilon}{2d} \bigg).
\end{eqnarray}
Now using the same arguments from the proof of Theorem 3.4 in Krizmani\'{c}~\cite{Kr22-2} we can bound the probability in first sum on the right-hand side of the last relation by
\begin{eqnarray*}
 \Pr \bigg( \bigvee_{i=1}^{n}  \bigg| \sum_{k=m-1}^{\infty} A_{k,j}Z_{i-k,n}^{(j)\leq} \bigg| > \frac{\epsilon}{2d} \bigg) &\leq&
 \Big(1+  \frac{(2d)^{\varphi}}{\epsilon^{\varphi}}\,n \mathrm{E}|Z_{1,n}^{(j)\leq}|^{\varphi} \Big) \sum_{k=m-1}^{\infty} \mathrm{E}|A_{k,j}|\\[0.2em]
& \leq & 3 \Big(1+  \frac{(2d)^{\varphi}}{\epsilon^{\varphi}}\,n \mathrm{E}|Z_{1,n}^{(j)\leq}|^{\varphi} \Big) \sum_{k=m-1}^{\infty} \mathrm{E}\|C_{k}\|
\end{eqnarray*}
for some $\varphi > \alpha$,
and the probability in the second sum by
$$ \Pr \bigg( \bigvee_{i=1}^{n} \bigg| \sum_{k=m-1}^{\infty} A_{k,j}Z_{i-k,n}^{(j)>} \bigg| > \frac{\epsilon}{2d} \bigg) \leq
 \left\{ \begin{array}{cl}
         \displaystyle \frac{6d}{\epsilon}\,n \mathrm{E} |Z_{1,n}^{(j)>}| \sum_{k=m-1}^{\infty} \mathrm{E} \|C_{k}\|, & \ \textrm{if} \ \alpha >1,\\[1.4em]
          \displaystyle 3 \frac{2^{\delta}}{\epsilon^{\delta} }\,n \mathrm{E}|Z_{1,n}^{(j)>}|^{\delta} \sum_{k=m-1}^{\infty} \mathrm{E} \|C_{k}\|^{\delta}, & \ \textrm{if} \ \alpha=1.
 \end{array}\right.$$
By Karamata's theorem and (\ref{eq:niz}) we obtain
$$ \limsup_{n \to \infty}\,n \mathrm{E}|Z^{(j)\leq}_{1,n}|^{\varphi} \leq \limsup_{n \to \infty} \frac{\mathrm{E}(\|Z_{1}\|^{\varphi}1_{\{ \|Z_{1}\| \leq a_{n}\}})}{a_{n}^{\varphi} \Pr(\|Z_{1}\| > a_{n})}\,n \Pr(\|Z_{1}\| > a_{n}) = \frac{\alpha}{\varphi - \alpha},$$
and for $\alpha >1$
$$ \limsup_{n \to \infty}\,n \mathrm{E}|Z^{(j)>}_{1,n}| \leq \limsup_{n \to \infty} \frac{\mathrm{E}(\|Z_{1}\|1_{\{ \|Z_{1}\| > a_{n}\}})}{a_{n} \Pr(\|Z_{1}\| > a_{n})}\,n \Pr(\|Z_{1}\| > a_{n}) = \frac{\alpha}{\alpha - 1}.$$
Similarly for $\alpha=1$ we have
$$ \limsup_{n \to \infty}\,n\mathrm{E}|Z^{(j)>}_{1,n}|^{\delta} \leq \frac{\alpha}{\alpha - \delta}.$$
Therefore we conclude
\begin{eqnarray*}
\limsup_{n \to \infty} \Pr \bigg( \bigvee_{i=1}^{n} \sum_{j=1}^{d} \bigg| \sum_{k=m-1}^{\infty}\frac{A_{k,j}Z_{i-k}^{(j)}}{a_{n}} \bigg| > \epsilon \bigg)&&\\[0.4em]
 & \hspace*{-28em} \leq& \hspace*{-14em}
 \left\{ \begin{array}{cl}
   \displaystyle A \sum_{k=m-1}^{\infty} \mathrm{E} \|C_{k}\|, & \quad \textrm{if} \ \alpha >1,\\[1.2em]
   \displaystyle A \bigg( \sum_{k=m-1}^{\infty} \mathrm{E} \|C_{k}\|+ \sum_{k=m-1}^{\infty} \mathrm{E}\|C_{k}\|^{\delta} \bigg), & \quad \textrm{if} \ \alpha=1,
\end{array}\right.
\end{eqnarray*}
for some constant $A$.
Letting $m \to \infty$ we see that conditions (\ref{e:momcondr}) and (\ref{eq:infmaTK3}) yield
$$ \lim_{m \to \infty} \limsup_{n \to \infty} \Pr \bigg[ \bigvee_{i=1}^{n} \sum_{j=1}^{d} \bigg| \sum_{k=m-1}^{\infty}\frac{A_{k,j}Z_{i-k}^{(j)}}{a_{n}} \bigg| > \epsilon \bigg] =0.$$
\end{proof}

\section*{Acknowledgements}
 This work has been supported in part by University of Rijeka research grants uniri-prirod-18-9, uniri-pr-prirod-19-16 and uniri-iskusni-prirod-23-98, and by Croatian Science Foundation under the project IP-2019-04-1239.


\end{document}